\documentclass{article}
\usepackage{iclr2027_conference,times}
\usepackage{amsmath,amssymb,amsthm,mathtools,booktabs,graphicx}
\usepackage{microtype,float,needspace}
\usepackage[hidelinks]{hyperref}
\hypersetup{pdftitle={Convergence Analysis of STORM Under Different Geometries},pdfauthor={Wei Jiang, Dingzhi Yu, Sifan Yang, Wenhao Yang, Zechao Li, Lijun Zhang}}
\newcommand{\E}{\mathbb E}
\newcommand{\R}{\mathbb R}
\newcommand{\norm}[1]{\left\|#1\right\|}
\newcommand{\inner}[2]{\left\langle#1,#2\right\rangle}
\newcommand{\ip}[2]{\inner{#1}{#2}}
\newcommand{\Ls}{\mathcal L}
\newcommand{\cF}{\mathcal F}

\newtheorem{theorem}{Theorem}
\newtheorem{lemma}[theorem]{Lemma}

\newtheorem{corollary}[theorem]{Corollary}
\theoremstyle{definition}
\newtheorem{assumption}{Assumption}
\theoremstyle{remark}

\allowdisplaybreaks[1]
\title{Convergence Analysis of STORM\\ Under Different Geometries}
\author{
{Wei Jiang}$^1$ \quad
{Yibo Wang}$^{2}$ \quad
{Wenhao Yang}$^{2}$ \quad
{Rui Yan}$^{1}$ \quad
{Lijun Zhang}$^2$ \quad
{Zechao Li}$^{1}$ \\[0.5em]
$^1$School of Computer Science and Engineering, \\
Nanjing University of Science and Technology, Nanjing, China \\
$^2$School of Artificial Intelligence, Nanjing University, Nanjing, China \\}
\iclrfinalcopy
\begin{document}
\maketitle

\begin{abstract}
Stochastic recursive momentum (STORM) achieves fast convergence for
nonconvex optimization via the variance reduction effect, but existing analyses  rely on the strong average smoothness assumption. In this paper, we study the convergence of STORM for different objectives without average smoothness.
We first revisit the results under average smoothness, obtaining the  $O(T^{-1/3})$ bound for nonconvex objectives and the
$O(\sigma^2/(\mu T))$ bound for last-iterate output under the
$\mu$-Polyak--\L{}ojasiewicz~(PL) condition. Without average smoothness, we design an auxiliary sequence and compare the STORM update with it in the analysis. 
With the help of this sequence, we prove that STORM still attains an $O(T^{-1/4})$ rate for nonconvex objectives, which is optimal under standard smoothness.
For convex and $\lambda$-strongly convex objectives, we further prove
averaged and last-iterate bounds with optimal rates of
$O(\sigma R/\sqrt T)$ and $O(\sigma^2/(\lambda T))$, respectively.
All the obtained results use the same STORM recursion with different hyperparameter choices. 
\end{abstract}

\section{Introduction}
\label{sec:introduction}
In this paper, we consider the stochastic optimization problem
\begin{equation}
 \min_{x\in\R^d} F(x),\qquad F(x)=\E_{\xi\sim\mathcal D}[f(x;\xi)],
 \label{eq:problem}
\end{equation}
where function $F(\cdot)$ is differentiable and $L$-smooth. We have access to an unbiased
stochastic gradient $g(x,\xi)$ with bounded variance such that
$\E_\xi g(x,\xi)=\nabla F(x)$ and
$\E_\xi\norm{g(x,\xi)-\nabla F(x)}^2\le\sigma^2$.
We measure the convergence rate by the expected gradient norm $\E\norm{\nabla F(x)}$ for the general nonconvex functions and by the
expected function gap $\E  [F(x)] - F^*$ for the other cases.

For nonconvex objectives, stochastic gradient descent (SGD) attains
an optimal $O(T^{-1/4})$ convergence rate under the smoothness assumption
\citep{ghadimi2013stochastic}. Variance reduction can further improve this
rate to $O(T^{-1/3})$ under the
additional average smoothness condition \citep{fang2018spider,arjevani2023lower}.
Among such methods, STORM \citep{cutkosky2019storm} is attractive
because it uses a constant batch size
per iteration and does not require periodically computing gradients with large batches.
Starting from $d_0=g(x_0,\xi_0)$, it performs
\begin{align}
 x_{t+1}&=x_t-\eta_t d_t,\label{eq:primal}\\
 d_{t+1}&=g(x_{t+1},\xi_{t+1})+(1-a_{t+1})
                   \bigl(d_t-g(x_t,\xi_{t+1})\bigr).
 \label{eq:estimator}
\end{align}
Here $\eta_t$ is the stepsize and $a_{t+1}$ is the momentum parameter.

Although STORM achieves the optimal rate for nonconvex objectives under average smoothness, many issues remain to be investigated.
First, existing analysis relies on the stronger \emph{average smoothness} condition, which assumes that $\E_\xi\norm{g(x,\xi)-g(y,\xi)}^2 \le \Ls^2 \norm{x-y}^2$. It remains unknown whether STORM can achieve the same $O(T^{-1/4})$ rate as SGD under standard smoothness alone. Second, existing guarantees do not establish optimal convergence rates for convex or strongly convex objectives without the average smoothness assumption. 

In this paper, we analyze these questions for the original STORM framework. First, we revisit the previous analysis under average smoothness and obtain the following results through a simpler and more transparent analysis. We do not regard the average smoothness results as the main contribution.
\begin{itemize}
 \item \underline{With average smoothness}, we first improve the previous analysis of STORM, achieving the $O(T^{-1/3})$ rate for nonconvex functions. Compared with previous analysis, we remove the logarithmic factor in the bound and avoid using large batches in the first iteration.
 \item \underline{With average smoothness and the $\mu$-PL condition},
 we obtain a \emph{last-iterate} bound with the convergence rate of
 $O(\sigma^2/(\mu T))$ via a two-phase stepsize design, allowing exponential decay of the initialization term.
\end{itemize}
Then, we investigate the convergence of STORM without the average smoothness condition, achieving the following results, which are the main contributions of this paper.
\begin{itemize}
 \item \underline{With standard smoothness}, we prove the $O(T^{-1/4})$ convergence rate for STORM, which matches the classical SGD rate and is optimal under standard assumptions.
 \item \underline{For smooth and convex objectives}, we obtain the weighted-mean and last-iterate bounds of order
 $\mathcal{O}(LR^2/T+\sigma R/\sqrt T)$, where $R$ is the initial distance to a minimizer. 
 \item \underline{For smooth and $\lambda$-strongly convex objectives}, we demonstrate that the weighted-mean and last-iterate bounds can be further improved to
 $\mathcal{O}(\sigma^2/(\lambda T))$.
\end{itemize}
Table~\ref{tab:results} summarizes the above convergence guarantees, which are all optimal in terms of the iteration number $T$ under the corresponding settings.

\begin{table}[t]
\caption{Convergence of STORM under different settings.
All rows assume an $L$-smooth objective and an unbiased oracle with variance at most $\sigma^2$. AS denotes average smoothness.
The nonconvex rows measure the expected gradient norm, while the other rows measure the expected function gap. Only the dominant term is presented.}
\label{tab:results}
\centering
\begin{tabular}{llll}
\toprule
Objective & Assumption & Output & Convergence rate\\
\midrule
Nonconvex & AS & Uniform random & $T^{-1/3}$\\
Nonconvex & $-$ & Uniform random & $T^{-1/4}$\\
$\mu$-PL & AS & Last iterate & $\sigma^2/(\mu T)$\\
Convex & $-$ & Weighted mean; last iterate& $\sigma R/\sqrt T$\\
$\lambda$-strongly convex & $-$ & Weighted mean; last iterate& $\sigma^2/(\lambda T)$\\
\bottomrule
\end{tabular}
\end{table}

\section{Related work}
\label{sec:related}
\paragraph{Variance reduction for nonconvex objectives.}
SARAH \citep{nguyen2017sarah} first introduced a recursive estimator based on successive stochastic
gradient differences. SPIDER \citep{fang2018spider} then uses this principle
to achieve $O(\varepsilon^{-3})$ stochastic gradient complexity and
$O(n+\sqrt n\,\varepsilon^{-2})$ finite-sum complexity for ensuring
$\E\norm{\nabla F(x)}\le\varepsilon$, where $n$ is the number of component
functions. These guarantees assume average smoothness. SPIDER periodically
refreshes its estimator using a large batch or the full gradient.
STORM \citep{cutkosky2019storm} instead incorporates a fresh-gradient difference term at every step,
removing these periodic refreshes. However, the original STORM analysis contains logarithmic factors in the rate.
In particular, its guarantee is
$O(\log(T)T^{-2/3})$ in expected squared gradient norm, hence
$O(\sqrt{\log(T)}T^{-1/3})$ in expected norm.\footnote{In
\citet{cutkosky2019storm}, the last simplification in Theorem~1's
proof drops a factor $\sqrt M$ uncarefully, where $M=O(\log(T+2))$. Furthermore, their Appendix~B, Theorem~2, explicitly retains
the logarithmic factor for the nonadaptive guarantee.} Other variants of STORM have also been investigated in the literature. STORM+ 
\citep{levy2021stormplus} obtains
$O(T^{-1/3})$ with adaptive parameters under the extra assumption of 
bounded gradients and bounded objective value. Ada-STORM \citep{jiang2024adastorm} attains the same rates without additional assumptions but with a large initial batchsize on the order $\mathcal{O}(T^{1/3})$.

\paragraph{STORM-based methods for convex optimization.}
The STORM estimator has also been studied beyond nonconvex
optimization.
\cite{zhang2020onesample} combine the STORM estimator with
Frank--Wolfe updates and obtain an $O(T^{-1/2})$ expected
function-gap bound for convex objectives.
However, their analysis requires a bounded feasible set, samplewise smoothness, and bounded stochastic gradients.
More recently, \citet{dahan2025double} combine STORM with
anytime averaging in $\mu^2$-SGD and attain the 
$O(T^{-1/2})$ rate.
Their method evaluates gradients at running averages of a
separate projected sequence, and its analysis uses samplewise
smoothness and a bounded feasible set.

\paragraph{Variance reduction under the $\mu$-PL condition.}
The variance reduction technique PAGE \citep{li2021page} obtains an expected function gap of at most $\varepsilon$ with
$O((B+\sqrt B\,\Ls/\mu)\log(2\Delta_0/\varepsilon))$
stochastic gradient evaluations, where
$B=\lceil2\sigma^2/(\mu\varepsilon)\rceil$. This result assumes the average smoothness and bounded variance, and its refresh batch depends on the target accuracy.
For compositional objectives, RECOVER \citep{qi2021recover} obtains the sample complexity
$O(\max\{\mu^{-3/2}\varepsilon^{-1/2},\mu^{-1}\varepsilon^{-1}\})$. 
For stochastic multi-level compositional optimization,
\citet{jiang2022multilevel} develop stage-wise SMVR using
STORM estimators.
For objectives satisfying the $\mu$-PL condition, their
stage-wise analysis gives complexity
$O(\mu^{-1}\varepsilon^{-1}
+\mu^{-3/2}\varepsilon^{-1/2})$,
which reduces to $O((\mu\varepsilon)^{-1})$
when $\varepsilon \le \mu$.

\section{Assumptions}
\label{sec:model}
We now state the assumptions for problem~\eqref{eq:problem}. The first
two assumptions are used throughout the paper and the remaining assumptions specify the setting of each theorem.

\begin{assumption}[Standard smoothness]
\label{ass:smooth}
The differentiable objective $F:\R^d\to\R$ has finite infimum
$F^*=\inf_xF(x)$ and satisfies
\begin{equation}
 \norm{\nabla F(x)-\nabla F(y)}\le L\norm{x-y}
 \quad\text{for all }x,y\in\R^d.
 \label{eq:population-smooth}
\end{equation}
\end{assumption}

\begin{assumption}[Stochastic oracle]
\label{ass:oracle}
The stochastic oracle $g$ is unbiased and has
bounded variance:
\begin{equation}
 \E_\xi g(x,\xi)=\nabla F(x),\qquad
 \E_\xi\norm{g(x,\xi)-\nabla F(x)}^2\le\sigma^2
 \quad\text{for all }x\in\R^d.
 \label{eq:bv}
\end{equation}
\end{assumption}

\begin{assumption}[Average smoothness]
\label{ass:average}
Stochastic gradient differences satisfy
\begin{equation}
 \E_\xi\norm{g(x,\xi)-g(y,\xi)}^2
       \le\Ls^2\norm{x-y}^2\quad\text{for all }x,y\in\R^d.
 \label{eq:mss}
\end{equation}
Note that this assumption is also called mean-square smoothness. 
By Jensen's inequality, Assumptions~\ref{ass:oracle}
and~\ref{ass:average} imply that $F$ is also $\Ls$-smooth, so we assume $L\le\Ls$ without loss of generality.
\end{assumption}

\begin{assumption}[PL condition]
\label{ass:pl}
For some $0<\mu\le L$, the objective satisfies
\begin{equation}
 \tfrac12\norm{\nabla F(x)}^2\ge\mu(F(x)-F^*)
 \quad\text{for all }x\in\R^d.
 \label{eq:pl}
\end{equation}
\end{assumption}

\begin{assumption}[Convexity]
\label{ass:convex}
The objective $F$ has a minimizer $x^*$ and satisfies
\begin{equation}
 F(y)\ge F(x)+\inner{\nabla F(x)}{y-x}
 \quad\text{for all }x,y\in\R^d.
 \label{eq:convexity}
\end{equation}
\end{assumption}

\begin{assumption}[Strong convexity]
\label{ass:strong}
The objective $F$ has a minimizer $x^*$ and  satisfies
\begin{equation}
 F(y)\ge F(x)+\inner{\nabla F(x)}{y-x}
                         +\tfrac\lambda2\norm{y-x}^2
 \quad\text{for all }x,y\in\R^d.
 \label{eq:strong-convexity}
\end{equation}
Strong convexity also implies the PL condition with $\mu=\lambda$, and the converse does not hold. 
\end{assumption}
We initialize at $x_0$, and define $\Delta_0=F(x_0)-F^*$,
 $R=\norm{x_0-x^*}$ and $\kappa=L/\lambda$ when applicable.

\section{Convergence under Average Smoothness}
\label{sec:nonconvex}
We first revisit the convergence results under average smoothness~(Assumption~\ref{ass:average}). We begin
with the general nonconvex objectives, then use the PL condition to obtain the function-gap guarantee.

\subsection{Nonconvex Objectives}
For the STORM algorithm, the estimator $d_t$  tracks the true function gradient $\nabla F(x_t)$, and we define 
$q_t=\E\norm{d_t-\nabla F(x_t)}^2$ for its mean squared error. Average smoothness bounds
the movement error by a term proportional to
$\Ls^2\eta_t^2\E\norm{d_t}^2$ via  $\E\norm{g(x_{t+1},\xi_{t+1})-g(x_t,\xi_{t+1})}^2\leq \Ls^2\E\norm{x_{t+1}-x_t}^2$. The descent inequality also contains a
negative term involving $\E\norm{d_t}^2$. To cancel them, we
use the potential
$\Psi_t=\E[F(x_t)-F^*]+\eta q_t/(2a_t)$ to obtain
\[
 \Psi_{t+1}-\Psi_t\le
 -\frac\eta2\E\norm{\nabla F(x_t)}^2+\eta a_{t+1}\sigma^2.
\]
The sum of $a_{t+1}=(t+3)^{-2/3}$ is $O(T^{1/3})$, which leads
to the following bound.

\begin{theorem}
\label{thm:nonconvex}
Under Assumptions~\ref{ass:smooth}, \ref{ass:oracle}, and \ref{ass:average}, 
setting
$\eta_t=\eta=\frac1{2\Ls(T+2)^{1/3}}$, $
 a_t=(t+2)^{-2/3}$, we have the following guarantee
\begin{align}
 \frac1T\sum_{t=0}^{T-1}\E\norm{\nabla F(x_t)}^2
 \le{}&\frac{4\Ls\Delta_0(T+2)^{1/3}}T
          +\frac{2^{2/3}\sigma^2}T+\frac{6\sigma^2}{T}
                 \bigl((T+2)^{1/3}-2^{1/3}\bigr).
\label{eq:ncbound}\end{align}
For $I$ uniform on $\{0,\ldots,T-1\}$ and independent of the trajectory,
\[
 \E\norm{\nabla F(x_I)}
 =O(\sqrt{\Ls\Delta_0+\sigma^2}\,T^{-1/3}).
\]
When
$\sigma=0$, the alternative choice $\eta=1/(4\Ls)$ and $a_t=1/4$
gives 
\[
 \E\norm{\nabla F(x_I)}
 =O(\sqrt{\Ls\Delta_0}/\sqrt T).
\]
\end{theorem}
\paragraph{Remark.}
Theorem~\ref{thm:nonconvex} gives an $O(T^{-1/3})$ rate and
the corresponding gradient complexity is $O(\varepsilon^{-3})$, matching the lower bound for nonconvex objectives with average smoothness 
\citep{arjevani2023lower}.
The theorem removes the logarithmic factor in the original STORM
guarantee \citep{cutkosky2019storm} and avoids using a large initial batch on the order of $\mathcal{O}(T^{1/3})$ \citep{jiang2024adastorm}. 

\subsection{Objectives Satisfying the PL Condition}
\label{sec:pl-main}

The PL inequality converts the negative gradient-norm term in the
descent bound into a decrease proportional to the current function
gap, i.e., $-\norm{\nabla F(x_t)}^2\le -2\mu(F(x_t)-F^*)$. In this case, we set the potential
$V_t=\E[F(x_t)-F^*]+q_t/\mu$ and couple the parameters through
$a_{t+1}=\mu\eta_t$. The descent and tracking inequalities will then give
\begin{equation}
 V_{t+1}\le(1-\mu\eta_t)V_t+2\mu\eta_t^2\sigma^2
 \quad\text{when }\eta_t\le\mu/(8\Ls^2).
 \label{eq:plscalar}
\end{equation}
Then, we can prove the following theorem under the PL condition.
\begin{theorem}
\label{thm:pl}
\textbf{(Harmonic schedule)} Under Assumptions~\ref{ass:smooth}--\ref{ass:pl}, for
$s\ge16(\Ls/\mu)^2$, set $
 \eta_t=\frac2{\mu(t+s)}$, $ a_{t+1}=\frac2{t+s}$.
For every integer $T\ge0$, the last iterate satisfies
\begin{equation}
 \E[F(x_T)-F^*]\le
 \left(\frac{s}{T+s}\right)^2
       \left(\Delta_0+\frac{\sigma^2}{\mu}\right)
       +\frac{8\sigma^2}{\mu(T+s)}.
 \label{eq:plbound}
\end{equation}
\textbf{(Two-phase schedule)} Alternatively, let $m=\lfloor T/2\rfloor$ and use
\begin{equation}
 \mu\eta_t=a_{t+1}=
 \begin{cases}2/s,&0\le t<m,\\2/(s+t-m),&m\le t<T.\end{cases}
 \label{eq:pl-two-phase}
\end{equation}
Without resetting either state, this gives
\begin{equation}
 \E[F(x_T)-F^*]\le
 e^{-2m/s}\left(\Delta_0+\frac{\sigma^2}{\mu}\right)
       +\frac{24\sigma^2}{\mu(T+2s)}.
 \label{eq:plphase}
\end{equation}
\end{theorem}

\paragraph{Remark.}
The first hyperparameter schedule gives quadratic decay of the initialization
term and the stochastic order $\sigma^2/(\mu T)$. Then,
the alternative two-phase schedule first reduces the initial error geometrically
and then decreases the stepsize to reduce the stochastic error.
It carries $(x_t,d_t)$ between phases without restarting and introduces no extra non-constant factor in the leading stochastic term. 

The dominant order $O(1/(\mu T))$ is unimprovable, since lower bounds for strongly convex functions also hold for PL objectives
\citep{agarwal2012information}. One important feature of our results is that the same estimator runs through both phases with one 
sample per step, and the rate is for the last iterate. 
\paragraph{Core of the analysis.}
We have already proved $V_{t+1}\le(1-\mu\eta_t)V_t+2\mu\eta_t^2\sigma^2$ in the previous analysis. When using $\mu\eta_t=2/(s+t)$, we have:
\[
 \prod_{j=k+1}^{T-1}\left(1-\frac2{j+s}\right)
 =\frac{(k+s-1)(k+s)}{(T+s-2)(T+s-1)}.
\]
The initial error is thus multiplied by
$O(s^2/(T+s)^2)$, while each noise term's $(k+s)^{-2}$ factor
is canceled, up to constants, by its product weight $(k+s-1)(k+s)$. Their sum gives
$O(\sigma^2/[\mu(T+s)])$. In the first phase of the alternative
schedule, by using a constant learning rate $\mu\eta_t=2/s$, the inequality $(1-2/s)^m\le e^{-2m/s}$ gives exponential decay.

\section{Convergence without Average Smoothness}
\label{sec:standard}
Without the average smoothness, nearby query points may not
have nearby stochastic gradients. The main difficulty is the
conditional bias of the estimator $d_t$. Denote $e_t=d_t-\nabla F(x_t)$. We
separate this inherited error from a centered increment $\nu_t$ via
$e_t=\beta e_{t-1}+\nu_t$.
Bounded variance can control $\E\norm{e_t}^2$, but this alone would leave a nonvanishing term in the usual descent inequality.

We handle the inherited error in two ways. For nonconvex objectives,
we shift the iterate by a multiple of $e_{t-1}$, so that the resulting
recursion contains only the centered noise. Standard smoothness then
controls the difference between the two gradient evaluation points.
For convex objectives, we compare STORM with an auxiliary 
sequence driven by the same $\nu_t$. The noise then cancels in their
difference, which can be controlled by nonexpansiveness, or by
contraction under strong convexity. 

\subsection{Nonconvex Objectives}
\label{sec:population-nonconvex}

The faster convergence guarantee in Theorem~\ref{thm:nonconvex}
uses the average smoothness condition. Without it, we show that STORM can still attain the optimal
$O(T^{-1/4})$ rate with standard smoothness. 
\begin{theorem}
\label{thm:population-nonconvex}
Under Assumptions~\ref{ass:smooth} and \ref{ass:oracle}, fixing any $\beta\in(0,1)$ and setting
$a_{t+1}=1-\beta$ and $0<\eta_t=\eta\le 1/(2L)$, we have the following guarantee
\begin{equation}
 \frac1T\sum_{t=0}^{T-1}\E\norm{\nabla F(x_t)}^2
 \le\frac{2\Delta_0}{\eta T}
       +\frac{1+\beta}{(1-\beta)^3}L\eta\sigma^2.
 \label{eq:population-nc-general}
\end{equation}
When $\Delta_0,\sigma>0$, choose
$\eta=\min\{1/(2L),\sqrt{\Delta_0/(6L\sigma^2T)}\}$ and $0<\beta\le1/2$. An independent uniform index
$I\in\{0,\ldots,T-1\}$ satisfies
\begin{equation}
 \E\norm{\nabla F(x_I)}
= \mathcal{O}\left(\sqrt{L\Delta_0}\,T^{-1/2} + (L\Delta_0\sigma^2)^{1/4}\,T^{-1/4}\right).
 \label{eq:population-nc-rate}
\end{equation}
When
$\sigma=0$, the alternative choice $\eta=1/(2L)$ and $0<\beta\le1/2$
gives 
\[
 \E\norm{\nabla F(x_I)}
 =O(\sqrt{L\Delta_0}/\sqrt{T}).
\]
\end{theorem}
\paragraph{Remark.}
When $\sigma>0$, the $O(T^{-1/4})$ rate matches SGD
\citep{ghadimi2013stochastic} and is optimal under standard smoothness \citep{arjevani2023lower}. The theorem also allows any fixed
$\beta\in(0,1/2]$.

\paragraph{Core of the analysis.} Define the estimation error
$e_t=d_t-\nabla F(x_t)$ and $e_{-1}=0$. The error obeys
$e_t=\beta e_{t-1}+\nu_t$, with $\nu_t=g(x_t,\xi_t)-\nabla F(x_t)
   -\beta\bigl(g(x_{t-1},\xi_t)-\nabla F(x_{t-1})\bigr)$ and $\nu_0=e_0$. Note that $\nu_t$ has conditional mean
zero and second moment at most $(1+\beta)^2\sigma^2$. Although
$e_t$ itself is conditionally biased, we can define the auxiliary iterations
$z_t=x_t-\eta\beta e_{t-1}/(1-\beta)$ which satisfy
\[
 z_{t+1}=z_t-\eta\nabla F(x_t)-\frac{\eta}{1-\beta}\nu_t.
\]
This removes the inherited error from the update.
Standard smoothness controls the difference between
$\nabla F(x_t)$ and $\nabla F(z_t)$, such that $\norm{\nabla F(x_t)-\nabla F(z_t)}^2 \le L^2\norm{x_t-z_t}^2 \leq \frac{L^2\eta^2\beta^2}{(1-\beta)^2}\norm{e_{t-1}}^2$. The shift thus has expected squared size $O(\eta^2\sigma^2)$ for fixed $\beta$. After deriving a descent argument, we can establish the bound~\eqref{eq:population-nc-general}. This argument uses the
smoothness of $F$ only, and it never bounds a difference of sampled
gradients by the distance between the query points.

\subsection{Convex Objectives}
\label{sec:convex-main}
We now use convexity to obtain function-gap guarantees without the
average smoothness assumption. The key observation is that standard
smoothness and convexity already provide useful stability
properties for the objective gradient map. These properties
allow us to compare STORM with a carefully designed auxiliary sequence, even when average smoothness is not assumed.
\begin{lemma}
\label{lem:geometry-main}
Under Assumptions~\ref{ass:smooth} and~\ref{ass:convex},
for every $0\le\eta\le1/L$ and all $x,y\in\R^d$,
\begin{equation}
 \norm{x-\eta\nabla F(x)-y+\eta\nabla F(y)}
 \le\norm{x-y}.
 \label{eq:nonexpansive-main}
\end{equation}
If strong convexity~(Assumption~\ref{ass:strong}) also holds, then
\begin{equation}
 \norm{x-\eta\nabla F(x)-y+\eta\nabla F(y)}
 \le(1-\lambda\eta)\norm{x-y}.
 \label{eq:contraction-main0}
\end{equation}
Under Assumption~\ref{ass:smooth} alone, even without convexity, we have
\begin{align}
 \norm{\nabla F(y)}^2
 &\le2L(F(y)-F^*),
 \label{eq:self-bound-main}\\
 F(x)-F^*
 &\le2(F(y)-F^*)+L\norm{x-y}^2.
 \label{eq:functioncompare}
\end{align}
\end{lemma}
\paragraph{Remark.}
Inequality~\eqref{eq:nonexpansive-main} shows that a gradient
step would not increase the distance between two points.
Under strong convexity, \eqref{eq:contraction-main0} can further give a contraction.
Both statements concern the exact gradient $\nabla F$ and therefore require no regularity of individual stochastic gradients, avoiding the requirement of average smoothness.
Once the distance between two trajectories is controlled, the inequality
\eqref{eq:functioncompare} transfers a function-gap guarantee
from one trajectory to the other.

\paragraph{Comparison with an auxiliary sequence.}
Let $\beta=1-a\in[0,1/2]$. We compare STORM with an auxiliary sequence that uses exact objective gradients
and the centered increments. This sequence is only used for analysis and requires no real computation.
\begin{lemma}
\label{lem:companion}
Under Assumptions~\ref{ass:smooth}, \ref{ass:oracle}, and
\ref{ass:convex}, let $0\le\beta\le1/2$ and choose $0<\eta_t\le1/L$.
Write $e_t=d_t-\nabla F(x_t)$ and $q_t=\E\norm{e_t}^2$.
Set $e_{-1}=0$ and $\nu_0=e_0$, and define 
\begin{equation}
 \nu_t=\zeta(x_t,\xi_t)-\beta\zeta(x_{t-1},\xi_t),
 \qquad e_t=\beta e_{t-1}+\nu_t,
 \qquad \zeta(x,\xi)=g(x,\xi)-\nabla F(x).
 \label{eq:innovation}
\end{equation}
For $\cF_t=\sigma(\xi_0,\ldots,\xi_{t-1})$, we have
\begin{equation}
 \E[\nu_t\mid\cF_t]=0,\qquad
 \E[\norm{\nu_t}^2\mid\cF_t]\le S^2:=(1+\beta)^2\sigma^2,
 \qquad q_t\le3\sigma^2.
 \label{eq:innovationbounds}
\end{equation}
The auxiliary sequence $y_0=x_0$,
\[
 y_{t+1}=y_t-\eta_t\bigl(\nabla F(y_t)+\nu_t\bigr)
\]
satisfies
\begin{equation}
 \sqrt{\E\norm{x_t-y_t}^2}
 \le\sqrt3\beta\sigma\sum_{j<t}\eta_j.
 \label{eq:convexcouple}
\end{equation}
Under Assumption~\ref{ass:strong}, this improves to
\begin{equation}
 \sqrt{\E\norm{x_t-y_t}^2}
 \le\sqrt3\beta\sigma/\lambda.
 \label{eq:sccouple}
\end{equation}
\end{lemma}
\paragraph{Core of the analysis.}
To see why the comparison works, subtract the two updates:
\[
 x_{t+1}-y_{t+1}
 =\bigl(x_t-\eta_t\nabla F(x_t)\bigr)
  -\bigl(y_t-\eta_t\nabla F(y_t)\bigr)
  -\eta_t\beta e_{t-1}.
\]
The shared increment $\nu_t$ cancels.
Applying~\eqref{eq:nonexpansive-main} and then the triangle
inequality in $L^2$ gives
\[
 \sqrt{\E\norm{x_{t+1}-y_{t+1}}^2}
 \le\sqrt{\E\norm{x_t-y_t}^2}
      +\eta_t\beta\sqrt{q_{t-1}}.
\]
Since $x_0=y_0$ and $q_{t-1}\le3\sigma^2$, summation gives
\eqref{eq:convexcouple}.
Under strong convexity, inequality \eqref{eq:contraction-main0} instead
multiplies the first term on the right by
$1-\lambda\eta_t$. With
$\E\norm{e_t}^2\le3\sigma^2$, this gives
\[
 \sqrt{\E\norm{x_{t+1}-y_{t+1}}^2}
 \le
 (1-\lambda\eta_t)\sqrt{\E\norm{x_t-y_t}^2}
 +\sqrt3\,\beta\sigma\eta_t,
\]
which further yields~\eqref{eq:sccouple}.
\paragraph{Remark.}Thus the inherited estimator error affects the trajectory
comparison through the smaller quantity $\beta e_{t-1}$.
Combining~\eqref{eq:convexcouple} with
\eqref{eq:functioncompare} gives
\[
 \E[F(x_t)-F^*]
 \le2\E[F(y_t)-F^*]
   +3L\beta^2\sigma^2
      \left(\sum_{j<t}\eta_j\right)^2.
\]
Choosing $\beta=O(1/(T+1))$ makes this additional error
compatible with the desired rates.

\paragraph{Stepsize blocks.}
Note that we also want to derive the rates for the last iterate output, and an average guarantee does not by itself control $F(x_T)$.
We therefore arrange the stepsizes so that a good value in an early
block controls values in the following shorter block. The blocks
shrink geometrically toward $T$, and the stepsize is halved at each
boundary. Larger early steps reduce the initial error and smaller late
steps limit the fresh noise injected near the returned point.
For any integer $T\ge1$, define
\begin{equation}
 m=\lfloor\log_2(T+1)\rfloor,\qquad
 U_i=T+1-\left\lfloor (T+1)/2^i\right\rfloor\ (0\le i\le m),
 \qquad U_{m+1}=T+1.
 \label{eq:rounded-boundaries}
\end{equation}
Each block $\{U_i,\ldots,U_{i+1}-1\}$ is nonempty,
and the last is
$\{T\}$. For example, when $T=12$,  their lengths are $7,3,2,1$.
The construction partitions the $T+1$ output, and it does not
restart the estimator.
To see how the blocks are compared, write
$h_t=\nabla F(y_t)+\nu_t$ and fix indices $p\le q$.
Expanding $\norm{y_{t+1}-y_p}^2$, taking expectations, and then using $\inner{\nabla F(y_t)}{y_t-y_p}\ge F(y_t)-F(y_p) $
gives
\begin{equation}
 2\sum_{t=p}^{q}\eta_t\E[F(y_t)-F(y_p)]
 \le\sum_{t=p}^{q}\eta_t^2\E\norm{h_t}^2.
 \label{eq:block-comparison-main}
\end{equation}
Lemma~\ref{lem:geometry-main} is useful again here.
Conditional centering and~\eqref{eq:self-bound-main} imply
\[
 \E\norm{h_t}^2
 =\E\norm{\nabla F(y_t)}^2+\E\norm{\nu_t}^2
 \le2L\,\E[F(y_t)-F^*]+S^2.
\]
Thus the right-hand side of~\eqref{eq:block-comparison-main}
can be controlled by function errors and the noise variance.
The stepsize restriction allows the function-error terms to
be absorbed into the resulting estimates, without assuming
globally bounded stochastic gradients.

We then compare successive blocks using a reference point
from the preceding block.
Halving the stepsizes makes the accumulated comparison errors
summable, and the final singleton block gives a last-iterate
guarantee.
\begin{theorem}
\label{thm:convex}
Under Assumptions~\ref{ass:smooth}, \ref{ass:oracle}, and \ref{ass:convex}, for every integer $T\ge1$, choose
\begin{equation}
 \beta=\frac1{T+2},\qquad
 \eta_t=2^{-i}\eta\quad(U_i\le t<U_{i+1}),\qquad
 0<\eta\le\frac1{32L}.
 \label{eq:convex-schedule}
\end{equation}
Define
$\bar x_T^\eta=\sum_{t=0}^T\eta_tx_t/\sum_{t=0}^T\eta_t$. Then
\begin{equation}
 \E[F(\bar x_T^\eta)-F^*]
 \le\frac{4R^2}{\eta(T+1)}+5\eta\sigma^2.
 \label{eq:convex-weighted-main}
\end{equation}
Both $z_T=x_T$ and $z_T=(T+1)^{-1}\sum_{t=0}^{T}x_t$ obey
\begin{equation}
 \E[F(z_T)-F^*]\le\frac{8\mathrm eR^2}{\eta(T+1)}
                                      +82\mathrm e\eta\sigma^2.
 \label{eq:convex-last-main}
\end{equation}
For $R,\sigma>0$, take
$\eta=\min\{1/(32L),\,2R/(\sqrt{41}\sigma\sqrt{T+1})\}$.
All three outputs $\widehat x_T$ satisfy
\begin{equation}
 \E[F(\widehat x_T)-F^*]\le
 \mathcal{O}\left(\frac{LR^2}{T}
       +\frac{\sigma R}{\sqrt{T}}\right).
 \label{eq:convex-rate-main}
\end{equation}
If $\sigma=0$, use $\eta=1/(32L)$ to obtain
\begin{equation}
 \E[F(\widehat x_T)-F^*]\le
 \mathcal{O}\left(\frac{LR^2}{T}\right).
\end{equation}
\end{theorem}

\paragraph{Remark.}
The dominant term $\sigma R/\sqrt T$ matches the SGD guarantees and the lower bound for convex objectives
\citep{agarwal2012information,stich2019unified,liu2026revisiting}. Compared with previous analysis of STORM, our rates remove the additional average smoothness assumption. Also note that we also establish a last-iterate bound in our analysis.

\subsection{Strongly Convex Objectives}
\label{sec:strong-main}
Under $\lambda$-strong convexity, the companion distance contracts.
With $\beta=1/(T+s+1)$, the function-gap comparison 
takes the form
\[
 \E[F(x_t)-F^*]
 \le
 2\E[F(y_t)-F^*]
 +\frac{3\kappa\sigma^2}{\lambda(T+s+1)^2}.
\]
Thus the error introduced by transferring the companion guarantee
back to STORM can be absorbed into the desired
$O(\sigma^2/(\lambda T))$ stochastic term.
For a carefully designed weighted mean, this permits exact
cancellation of consecutive distance terms. For the last point, we
combine this contraction with the rounded blocks in
\eqref{eq:rounded-boundaries}. The two results use different schedules.

\begin{theorem}[Weighted mean]
\label{thm:strong-mean-main}
Under Assumptions~\ref{ass:smooth}, \ref{ass:oracle}, and \ref{ass:strong}, for any $s\ge4\kappa$, use
\begin{equation}
 \beta=\frac1{T+s+1},\qquad \eta_t=\frac2{\lambda(t+s+1)},\qquad
 \bar x_T=\frac{2\sum_{t=0}^{T}(t+s)x_t}{(T+1)(T+2s)}.
 \label{eq:sc-mean-schedule}
\end{equation}
Then, we can ensure that
\begin{equation}
 \E[F(\bar x_T)-F^*]\le
 \frac{2\lambda s(s-1)R^2}{(T+1)(T+2s)}
       +\frac{19\sigma^2}{\lambda(T+s+1)}.
 \label{eq:strong-mean-main}
\end{equation}
The same bound holds for $x_I$, where $I$ has probabilities proportional to $t+s$.
\end{theorem}
Next, we present the guarantee for the last iterate below.
\begin{theorem}[Last iterate]
\label{thm:strong}
Under Assumptions~\ref{ass:smooth}, \ref{ass:oracle}, and \ref{ass:strong},
let $T\ge1$ be any integer, $s\ge4\kappa^2$, and $\beta=1/(T+s+1)$.
For the rounded blocks \eqref{eq:rounded-boundaries}, use
\begin{equation}
 \eta_t=2^{-i}\frac2{\lambda(t+s+1)}
 \quad(U_i\le t<U_{i+1}).
 \label{eq:sctaper}
\end{equation}
After $T$ primal updates, the last iterate satisfies
\begin{align}
 \E[F(x_T)-F^*]
 &\le\frac{288\lambda(s+1)^2R^2}{(T+1)^2}
                           +\frac{2451\sigma^2}{\lambda(T+1)}.
 \label{eq:scrate}
\end{align}
\end{theorem}
\paragraph{Remark.}
The dominant term $O(\sigma^2/(\lambda T))$ of the last iterate bound has the same order as the mean, which is optimal for strongly convex functions
\citep{agarwal2012information}. 
Compared with previous results, we remove the average smoothness assumption and establish a last-iterate bound.

\section{Numerical Experiments}
\label{sec:experiments}
We study three questions suggested by the theory: how the objective
geometry affects convergence, whether STORM converges when average
smoothness fails, and how the prescribed schedules behave for averaged
and last-iterate outputs. Our contribution is an analysis of the original
STORM algorithm and its empirical performance has already been studied
extensively in previous literature,
so we focus on experiments that isolate these questions
instead of repeating a broad benchmark evaluation. 

\paragraph{Experimental setup.}
We use the following four objectives with known minima and exact
gradients. An orthogonal change of coordinates $z=Qx$ gives
\begin{align*}
 F_{\rm nc}(x)&=\sum_{i=1}^{10}\left(1-\cos z_i+0.05z_i^2\right),&
 F_{\rm PL}(x)&=\frac12\sum_{i=1}^{10}(z_i+2\tanh z_i)^2,\\
 F_{\rm cv}(x)&=\frac12\sum_{i=1}^{10} \left(z_i^2-\log(1+z_i^2)\right),&
 F_{\rm sc}(x)&=\frac12\sum_{i=1}^{10}h_i z_i^2,
\end{align*}
where $h_i$ are equally spaced in $[1,4]$. 
For the average-smooth, we add independent standard Gaussian
noise to the coordinates of the gradient. For the experiments without
average smoothness, we use
\begin{equation}
 g(x;J,\phi)=\nabla F(x)+Q^\top\zeta(Qx;J,\phi),\qquad
 \zeta_i(z;J,\phi)=\sqrt2\sin(2^{J_i}z_i+\phi_i),
 \label{eq:roughoracle}
\end{equation}
where $\Pr(J_i=j)=2^{-j-1}$ and $\phi_i$ is uniform on $[0,2\pi]$. The oracle is unbiased and has
$\sigma^2=10$, but admits no finite average-smoothness constant.
We run $256$ independent trials.
All parameters are chosen from the stated analytical constants, without
tuning hyperparameters.

\paragraph{Nonconvex and PL objectives under average smoothness.}
For $F_{\rm nc}$, we use the exact schedule in
Theorem~\ref{thm:nonconvex} and measure
$T^{-1}\sum_{t=0}^{T-1}\|\nabla F(x_t)\|^2$, whose
reference order is $T^{-2/3}$. For $F_{\rm PL}$, we compare the
harmonic and two-phase schedules of Theorem~\ref{thm:pl} and report $F(x_T)-F^*$.
\begin{figure}[t]
\centering
\includegraphics[width=\textwidth]{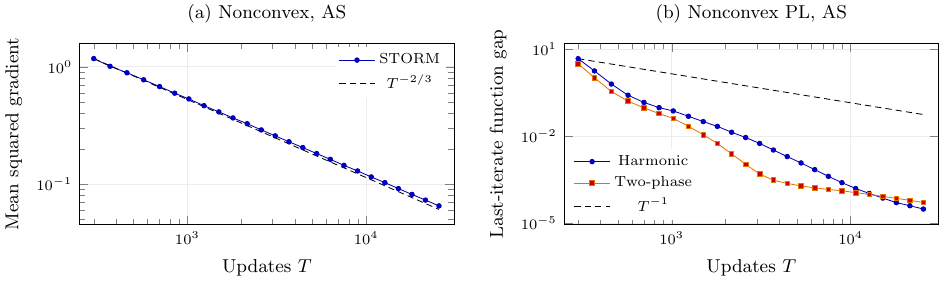}
\caption{Convergence under average smoothness (log--log scales).}
\label{fig:diagnostics}
\end{figure}
Figure~\ref{fig:diagnostics}(a) is roughly consistent with the
$T^{-2/3}$ reference. On the PL problem, the two-phase schedule gives a smaller
gap at smaller budgets, whereas the harmonic schedule
becomes better at the larger budgets.

\paragraph{Nonconvex objectives without average smoothness.}
For $F_{\rm nc}$, we replace the Gaussian oracle
by~\eqref{eq:roughoracle}. We test $\beta\in\{0.1,0.3,0.5\}$ with
the common stepsize prescribed in
Theorem~\ref{thm:population-nonconvex}. 
SGD uses the same stepsize as a reference. We measure the
average squared gradient. Here, the theoretical stochastic
order is $T^{-1/2}$, corresponding to $T^{-1/4}$ in gradient norm.
In Figure~\ref{fig:standard-diagnostics}(a), increasing $T$ by a
factor of $85$ reduces the squared-gradient criterion by about
$9.1$ for each tested correction coefficient, close to
$\sqrt{85}\approx9.2$. The three fixed-$\beta$ curves and SGD
are nearly indistinguishable at this scale.

\paragraph{Averaged and last iterates without average smoothness.}
For $F_{\rm cv}$, we use the rounded block schedule in
Theorem~\ref{thm:convex} and evaluate the last iterate, the uniform
mean, and the stepsize-weighted mean of the same trajectory.
For $F_{\rm sc}$, we run the two schedules in
Theorems~\ref{thm:strong-mean-main} and~\ref{thm:strong}
separately.
\begin{figure}[t]
\centering
\includegraphics[width=\textwidth]{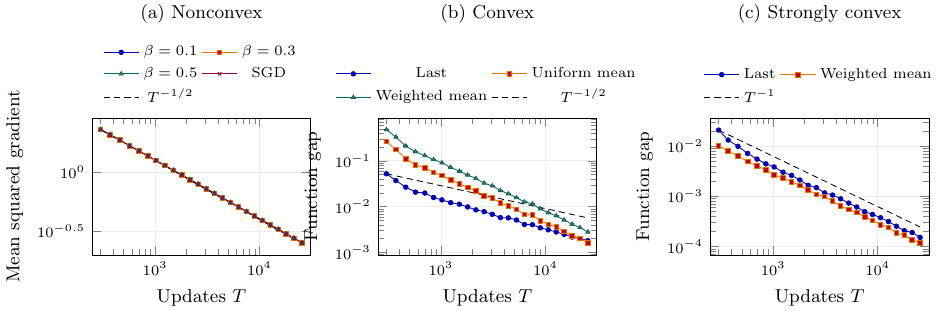}
\caption{Convergence without average smoothness (log--log scales).}
\label{fig:standard-diagnostics}
\end{figure}
In Figure~\ref{fig:standard-diagnostics}(b), the uniform mean
is slightly better than the last iterate at the largest
tested budget.
For strong convexity, both curves in
Figure~\ref{fig:standard-diagnostics}(c) follow the $T^{-1}$
reference closely at larger budgets.

Also note that our theoretical guarantees are worst-case upper bounds. A particular objective can converge faster than the worst-case rate in practice.

\section{Conclusion}

In this paper, we studied the convergence of STORM under different
objective geometries and smoothness assumptions.
Under average smoothness, we provided a direct analysis of the
logarithm-free nonconvex rate  and established a last-iterate
guarantee with an optimal stochastic term under the PL condition.
Under standard smoothness alone, STORM matches the nonconvex
SGD rate and attains optimal stochastic orders for both averaged
and last iterates in the convex and strongly convex settings.
These guarantees follow from different parameter choices within
the same STORM recursion, with a single initial sample and a
constant number of gradient evaluations per iteration.
\label{maintextend}

\bibliography{references}
\bibliographystyle{iclr2027_conference}
\clearpage
\appendix
\section{Proof of Theorem~\ref{thm:nonconvex}}
\label{app:nonconvex}
We prove Theorem~\ref{thm:nonconvex} by combining the descent
inequality with a bound on the gradient estimation error.
Let $\cF_t=\sigma(\xi_0,\ldots,\xi_{t-1})$. Define
\[
 e_t=d_t-\nabla F(x_t),\qquad q_t=\E\norm{e_t}^2,
 \qquad \Delta_t=\E[F(x_t)-F^*].
\]
Thus $\Delta_0=F(x_0)-F^*$ and $q_0\le\sigma^2$.

\begin{lemma}
\label{lem:tracking}
Under Assumptions~\ref{ass:smooth} and~\ref{ass:oracle},
for every deterministic $\eta_t\ge0$, the primal update satisfies
\begin{align}
 \Delta_{t+1}-\Delta_t
 \le{}&-\frac{\eta_t}{2}\E\norm{\nabla F(x_t)}^2
       -\frac{\eta_t(1-L\eta_t)}{2}\E\norm{d_t}^2
       +\frac{\eta_t}{2}q_t.
 \label{eq:descent}
\end{align}
If Assumption~\ref{ass:average} also holds and
$0\le a_{t+1}\le1$, then
\begin{equation}
 q_{t+1}\le(1-a_{t+1})^2q_t+2a_{t+1}^2\sigma^2
                       +2\Ls^2\eta_t^2\E\norm{d_t}^2.
 \label{eq:tracking}
\end{equation}
\end{lemma}

\begin{proof} Smoothness and $x_{t+1}=x_t-\eta_td_t$ give
\[
 F(x_{t+1})\le F(x_t)
 -\eta_t\ip{\nabla F(x_t)}{d_t}
 +\frac{L\eta_t^2}{2}\norm{d_t}^2.
\]
By the definition of $e_t$,
\[
 \norm{e_t}^2
 =\norm{d_t}^2+\norm{\nabla F(x_t)}^2
                      -2\ip{\nabla F(x_t)}{d_t}.
\]
Solving this identity for the inner product and substituting yields
\begin{align*}
 F(x_{t+1})-F(x_t)
 &\le-\frac{\eta_t}{2}\norm{\nabla F(x_t)}^2
      -\frac{\eta_t}{2}\norm{d_t}^2
      +\frac{\eta_t}{2}\norm{e_t}^2
      +\frac{L\eta_t^2}{2}\norm{d_t}^2\\
 &=-\frac{\eta_t}{2}\norm{\nabla F(x_t)}^2
      -\frac{\eta_t(1-L\eta_t)}{2}\norm{d_t}^2
      +\frac{\eta_t}{2}\norm{e_t}^2.
\end{align*}
Taking expectations proves \eqref{eq:descent}.

To prove the error recursion, abbreviate $a=a_{t+1}$ and write
$\zeta(x,\xi)=g(x,\xi)-\nabla F(x)$.
Subtracting $\nabla F(x_{t+1})$ from the estimator update gives the
exact identity
\begin{align*}
 e_{t+1}
 &=(1-a)e_t+\zeta(x_{t+1},\xi_{t+1})
                         -(1-a)\zeta(x_t,\xi_{t+1})\\
 &=(1-a)e_t+a\zeta(x_{t+1},\xi_{t+1})+(1-a)
       \bigl(\zeta(x_{t+1},\xi_{t+1})-\zeta(x_t,\xi_{t+1})\bigr).
\end{align*}
Conditional on $\cF_{t+1}$, the fresh-sample terms together have mean
zero, so their cross term with the fixed $e_t$ vanishes. We then use
$\norm{u+v}^2\le2\norm{u}^2+2\norm{v}^2$ to obtain
\begin{align*}
 \E[\norm{e_{t+1}}^2\mid\cF_{t+1}]
 &\le(1-a)^2\norm{e_t}^2
       +2a^2\E[\norm{\zeta(x_{t+1},\xi_{t+1})}^2\mid\cF_{t+1}]\\
 &\quad+2(1-a)^2\E\bigl[
       \norm{\zeta(x_{t+1},\xi_{t+1})-\zeta(x_t,\xi_{t+1})}^2
                           \mid\cF_{t+1}\bigr].
\end{align*}
The first conditional noise moment is at most $\sigma^2$ by
\eqref{eq:bv} and fresh-sample independence. For the second one,
centering a random vector subtracts the squared norm of its mean:
\begin{align*}
 &\E\bigl[
    \norm{\zeta(x_{t+1},\xi_{t+1})-\zeta(x_t,\xi_{t+1})}^2
                       \mid\cF_{t+1}\bigr]\\
 &\quad=\E\bigl[
    \norm{g(x_{t+1},\xi_{t+1})-g(x_t,\xi_{t+1})}^2
                       \mid\cF_{t+1}\bigr]
          -\norm{\nabla F(x_{t+1})-\nabla F(x_t)}^2\\
 &\quad\le\Ls^2\norm{x_{t+1}-x_t}^2
       =\Ls^2\eta_t^2\norm{d_t}^2.
\end{align*}
The inequality uses \eqref{eq:mss}. Substitution gives
\[
 \E[\norm{e_{t+1}}^2\mid\cF_{t+1}]
 \le(1-a)^2\norm{e_t}^2+2a^2\sigma^2
                  +2(1-a)^2\Ls^2\eta_t^2\norm{d_t}^2.
\]
Finally, $(1-a)^2\le1$ gives
\eqref{eq:tracking}.
\end{proof}

\begin{lemma}
\label{lem:certificate}
Under Assumptions~\ref{ass:smooth}, \ref{ass:oracle}, and
\ref{ass:average}, let
$\eta_t=\eta>0$ and $L\eta\le1/2$.
Choose an auxiliary weight $a_0>0$ and deterministic coefficients satisfying
\[
 0<a_{t+1}\le\frac12,\qquad
 a_{t+1}\ge4\Ls^2\eta^2,\qquad
 a_{t+1}^{-1}-a_t^{-1}\le\frac12.
\]
With
\[
 \Psi_t=\Delta_t+\frac{\eta}{2a_t}q_t,
\]
one has
\begin{equation}
 \Psi_{t+1}\le\Psi_t
       -\frac\eta2\E\norm{\nabla F(x_t)}^2
       +\eta a_{t+1}\sigma^2.
 \label{eq:certificate}
\end{equation}
\end{lemma}

\begin{proof}
Multiply \eqref{eq:tracking} by $\eta/(2a_{t+1})$ and add it to
\eqref{eq:descent}. This gives
\begin{align*}
 \Psi_{t+1}
 &\le\Delta_t-\frac\eta2\E\norm{\nabla F(x_t)}^2+\left(\frac\eta2+
             \frac{\eta(1-a_{t+1})^2}{2a_{t+1}}\right)q_t\\
 &\quad+\left(-\frac{\eta(1-L\eta)}2
                   +\frac{\Ls^2\eta^3}{a_{t+1}}\right)
                                      \E\norm{d_t}^2
        +\eta a_{t+1}\sigma^2.
\end{align*}
Subtract $\Psi_t=\Delta_t+\eta q_t/(2a_t)$.
The remaining coefficient of $q_t$ is
\begin{align*}
 \frac\eta2+\frac{\eta(1-a_{t+1})^2}{2a_{t+1}}
                     -\frac{\eta}{2a_t}
 &=\frac\eta2\left(
      1+a_{t+1}^{-1}-2+a_{t+1}-a_t^{-1}\right)\\
 &=\frac\eta2\left(
      a_{t+1}^{-1}-a_t^{-1}-1+a_{t+1}\right)\\
 &\le\frac\eta2\left(\frac12-1+\frac12\right)=0.
\end{align*}
The coefficient of $\E\norm{d_t}^2$ is also nonpositive:
\[
 -\frac{\eta(1-L\eta)}2+\frac{\Ls^2\eta^3}{a_{t+1}}
 \le-\frac\eta4+\frac{\Ls^2\eta^3}{4\Ls^2\eta^2}=0.
\]
Dropping these two nonpositive terms proves \eqref{eq:certificate}.
\end{proof}

\begin{proof}[Proof of Theorem~\ref{thm:nonconvex}]
Fix an arbitrary integer $T\ge1$ and use:
\[
 \eta=\frac1{2\Ls(T+2)^{1/3}},\qquad
 a_t=(t+2)^{-2/3}.
\]
We first verify the conditions of Lemma~\ref{lem:certificate} for
$t=0,\ldots,T-1$. Since $L\le\Ls$,
\[
 L\eta\le\frac1{2(T+2)^{1/3}}\le\frac12.
\]
The mixing sequence decreases, and hence
\[
 a_{t+1}\le a_1=3^{-2/3}<\frac12,
 \qquad
 a_{t+1}\ge a_T=(T+2)^{-2/3}=4\Ls^2\eta^2.
\]
Because $u\mapsto u^{2/3}$ is concave, its unit increments
decrease as $u$ increases. Therefore
\[
 a_{t+1}^{-1}-a_t^{-1}
 =(t+3)^{2/3}-(t+2)^{2/3}
 \le3^{2/3}-2^{2/3}<\frac12.
\]
Thus, all conditions of Lemma~\ref{lem:certificate} hold.
Rearrange \eqref{eq:certificate} and sum over the $T$ steps:
\begin{align*}
 \frac\eta2\sum_{t=0}^{T-1}\E\norm{\nabla F(x_t)}^2
 &\le\sum_{t=0}^{T-1}(\Psi_t-\Psi_{t+1})
                        +\eta\sigma^2\sum_{t=0}^{T-1}a_{t+1}\\
 &=\Psi_0-\Psi_T+\eta\sigma^2\sum_{t=0}^{T-1}a_{t+1}\\
 &\le\Delta_0+\frac{\eta\sigma^2}{2a_0}
                        +\eta\sigma^2\sum_{t=0}^{T-1}a_{t+1}.
\end{align*}
Here $\Psi_T\ge0$ because $F(x_T)\ge F^*$ and $q_T\ge0$,
and the initial variance bound gives $q_0\le\sigma^2$.
Dividing by $\eta T/2$ yields
\[
 \frac1T\sum_{t=0}^{T-1}\E\norm{\nabla F(x_t)}^2
 \le\frac{2\Delta_0}{\eta T}
        +\frac{\sigma^2}{a_0T}
        +\frac{2\sigma^2}{T}\sum_{t=0}^{T-1}(t+3)^{-2/3}.
\]

Since $u^{-2/3}$ is decreasing, integrating over each interval
$[t+2,t+3]$ gives
\begin{align*}
 \sum_{t=0}^{T-1}(t+3)^{-2/3}
 &\le\sum_{t=0}^{T-1}\int_{t+2}^{t+3}u^{-2/3}\,du\\
 &=\int_2^{T+2}u^{-2/3}\,du\\
 &=3\bigl((T+2)^{1/3}-2^{1/3}\bigr).
\end{align*}
Substitute this bound, $\eta^{-1}=2\Ls(T+2)^{1/3}$, and
$a_0^{-1}=2^{2/3}$ to obtain
\begin{align*}
 \frac1T\sum_{t=0}^{T-1}\E\norm{\nabla F(x_t)}^2
 \le{}&\frac{4\Ls\Delta_0(T+2)^{1/3}}T
       +\frac{2^{2/3}\sigma^2}T+\frac{6\sigma^2}{T}
                      \bigl((T+2)^{1/3}-2^{1/3}\bigr),
\end{align*}
which is \eqref{eq:ncbound}.

Finally, let $I$ be uniform on $\{0,\ldots,T-1\}$ and independent of
the trajectory. Then
\[
 \E\norm{\nabla F(x_I)}^2
 =\frac1T\sum_{t=0}^{T-1}\E\norm{\nabla F(x_t)}^2.
\]
Cauchy--Schwarz gives
$\E\norm{\nabla F(x_I)}\le
(\E\norm{\nabla F(x_I)}^2)^{1/2}$.
Since $T+2\le3T$ for $T\ge1$, the displayed second-moment
bound is $O((\Ls\Delta_0+\sigma^2)T^{-2/3})$, proving the
stated expected-norm rate.
\end{proof}

\emph{The noiseless case.}

If $\sigma=0$, take $\eta=1/(4\Ls)$ and $a_t=1/4$ for every
$t\ge0$, including the potential weight $a_0$.
Then $L\eta\le1/4$, $a_{t+1}=4\Ls^2\eta^2=1/4$, and
$a_{t+1}^{-1}-a_t^{-1}=0$, so the certificate applies.
Since $q_0=0$ and the noise term vanishes, summing it gives
\[
 \frac1T\sum_{t=0}^{T-1}\E\norm{\nabla F(x_t)}^2
 \le\frac{2\Delta_0}{\eta T}=\frac{8\Ls\Delta_0}{T}.
\]

\section{Proof of Theorem~\ref{thm:pl}}
\label{app:pl}
\begin{proof}
Retain the notation
$e_t=d_t-\nabla F(x_t)$, $q_t=\E\norm{e_t}^2$, and
$\Delta_t=\E[F(x_t)-F^*]$.
The potential in this proof has a constant estimator-error weight:
\[
 V_t=\Delta_t+\frac{q_t}{\mu}.
\]
Also note that $\Delta_t\le V_t$ and
$V_0\le\Delta_0+\sigma^2/\mu$.

First, consider any one step satisfying
\[
 a_{t+1}=\mu\eta_t,\qquad
 0<\mu\eta_t\le\frac12,\qquad
 \eta_t\le\min\left\{\frac1{2L},\frac{\mu}{8\Ls^2}\right\}.
\]
Add \eqref{eq:tracking} divided by $\mu$ to
\eqref{eq:descent}. The resulting inequality is
\begin{align*}
 V_{t+1}
 &\le\Delta_t-\frac{\eta_t}{2}\E\norm{\nabla F(x_t)}^2+\left(\frac{(1-\mu\eta_t)^2}{\mu}
                          +\frac{\eta_t}{2}\right)q_t\\
 &\quad+\left(-\frac{\eta_t(1-L\eta_t)}2
                          +\frac{2\Ls^2\eta_t^2}{\mu}\right)
                                      \E\norm{d_t}^2
       +2\mu\eta_t^2\sigma^2.
\end{align*}
By the PL inequality,
\[
 \Delta_t-\frac{\eta_t}{2}\E\norm{\nabla F(x_t)}^2
 \le(1-\mu\eta_t)\Delta_t.
\]
For the coefficient of $q_t/\mu$, direct expansion gives
\begin{align*}
 (1-\mu\eta_t)^2+\frac{\mu\eta_t}{2}
 &=1-\frac32\mu\eta_t+(\mu\eta_t)^2\\
 &=1-\mu\eta_t+
                 \mu\eta_t\left(\mu\eta_t-\frac12\right)\\
 &\le1-\mu\eta_t.
\end{align*}
Finally, $L\eta_t\le1/2$ and
$\eta_t\le\mu/(8\Ls^2)$ imply
\[
 -\frac{\eta_t(1-L\eta_t)}2
                          +\frac{2\Ls^2\eta_t^2}{\mu}
 \le-\frac{\eta_t}{4}
        +\eta_t\frac{2\Ls^2}{\mu}\frac{\mu}{8\Ls^2}=0.
\]
Since $\E\norm{d_t}^2\ge0$, that term can be discarded.
Combining the other two coefficients gives
\begin{equation}
 V_{t+1}\le(1-\mu\eta_t)V_t+2\mu\eta_t^2\sigma^2.
 \label{eq:plscalar-app}
\end{equation}

For either schedule in the theorem, the largest stepsize is
$2/(\mu s)$. Since $s\ge16(\Ls/\mu)^2$,
\[
 \frac2{\mu s}\le\frac{\mu}{8\Ls^2},\qquad
 \frac2s\le\frac{\mu^2}{8\Ls^2}\le\frac18,
 \qquad
 L\frac2{\mu s}\le\frac{L\mu}{8\Ls^2}\le\frac18.
\]
Thus \eqref{eq:plscalar-app} holds at every step of either
schedule. In particular, $s\ge16$, so the scalar factors used
below are positive.

\emph{The first schedule.}
Substituting $\eta_t=2/[\mu(t+s)]$ into \eqref{eq:plscalar-app} gives
\begin{equation}
 V_{t+1}\le\frac{t+s-2}{t+s}V_t
                   +\frac{8\sigma^2}{\mu(t+s)^2}.
 \label{eq:pl-reciprocal-recursion}
\end{equation}
The first two substitutions are
\begin{align*}
 V_1&\le\frac{s-2}{s}V_0+\frac{8\sigma^2}{\mu s^2},\\
 V_2&\le\frac{s-1}{s+1}\frac{s-2}{s}V_0
       +\frac{8\sigma^2}{\mu}
           \left(\frac{s-1}{(s+1)s^2}+\frac1{(s+1)^2}\right).
\end{align*}
Induction therefore gives
\begin{equation}\label{eq:pl-unrolled}
 V_T\le V_0\prod_{k=0}^{T-1}\frac{k+s-2}{k+s}
       +\frac{8\sigma^2}{\mu}\sum_{j=0}^{T-1}\frac1{(j+s)^2}
                 \prod_{k=j+1}^{T-1}\frac{k+s-2}{k+s}.
\end{equation}
For every $0\le u\le T$,
\begin{equation}\label{eq:pl-product}
 \prod_{k=u}^{T-1}\frac{k+s-2}{k+s}
 =\frac{(u+s-2)(u+s-1)}{(T+s-2)(T+s-1)}.
\end{equation}
Taking $u=0$ gives the initialization product.
Taking $u=j+1$ in each noise term gives
\begin{align*}
 \frac1{(j+s)^2}
                 \prod_{k=j+1}^{T-1}\frac{k+s-2}{k+s}
 &=\frac1{(j+s)^2}
      \frac{(j+s-1)(j+s)}{(T+s-2)(T+s-1)}\\
 &=\frac{(j+s-1)/(j+s)}{(T+s-2)(T+s-1)}.
\end{align*}
Consequently, \eqref{eq:pl-unrolled} becomes the explicit bound
\begin{equation}
 V_T\le
 \frac{(s-2)(s-1)V_0+
       \dfrac{8\sigma^2}{\mu}
                    \sum_{j=0}^{T-1}\dfrac{j+s-1}{j+s}}
      {(T+s-2)(T+s-1)}.
 \label{eq:pl-exact-scalar-bound}
\end{equation}

We now bound the initialization and noise terms separately.

For $r=1,2$, positivity of the denominators and $rT\ge0$ imply
\[
 \frac{s-r}{T+s-r}\le\frac{s}{T+s}.
\]
Indeed, after cross multiplication this is precisely
$(s-r)(T+s)\le s(T+s-r)$.
Multiplying the two inequalities gives
\[
 \frac{(s-2)(s-1)}{(T+s-2)(T+s-1)}
 \le\left(\frac{s}{T+s}\right)^2.
\]
Every summand $(j+s-1)/(j+s)$ in the noise numerator is at most
one. Furthermore,
\begin{align*}
 (T+s-2)(T+s-1)-T(T+s)=(s-3)(T+s)+2\ge0,
\end{align*}
because $s\ge16$. Hence
\[
 \frac{\sum_{j=0}^{T-1}(j+s-1)/(j+s)}{(T+s-2)(T+s-1)}
 \le\frac{T}{(T+s-2)(T+s-1)}\le\frac1{T+s}.
\]
Combining the two bounds yields
\begin{equation}
 V_T\le\left(\frac{s}{T+s}\right)^2V_0
                         +\frac{8\sigma^2}{\mu(T+s)}.
 \label{eq:pl-potential-bound}
\end{equation}
Using $\Delta_T\le V_T$ and
$V_0\le\Delta_0+\sigma^2/\mu$ proves \eqref{eq:plbound}.

\emph{The two-phase schedule.}

Now fix any integer $T\ge1$, put $m=\lfloor T/2\rfloor$, and
use \eqref{eq:pl-two-phase}. During the first $m$ updates,
\eqref{eq:plscalar-app} gives
\[
 V_{t+1}\le\left(1-\frac2s\right)V_t
                                    +\frac{8\sigma^2}{\mu s^2}.
\]
Repeated substitution gives
\begin{align*}
 V_m
 &\le\left(1-\frac2s\right)^mV_0
       +\frac{8\sigma^2}{\mu s^2}
                      \sum_{j=0}^{m-1}\left(1-\frac2s\right)^j\\
 &=\left(1-\frac2s\right)^mV_0
       +\frac{4\sigma^2}{\mu s}
                         \left[1-\left(1-\frac2s\right)^m\right]\\
 &\le e^{-2m/s}V_0+\frac{4\sigma^2}{\mu s}.
\end{align*}
The final line uses
$1-u\le e^{-u}$ and drops a factor between zero and one
from the noise term.

There remain $n=T-m$ updates. Their parameters are
$\eta_{m+j}=2/[\mu(s+j)]$ for $j=0,\ldots,n-1$.
Applying the bound \eqref{eq:pl-potential-bound}
with initial value $V_m$ gives
\[
 V_T\le\left(\frac{s}{s+n}\right)^2V_m
                                      +\frac{8\sigma^2}{\mu(s+n)}.
\]
The potential $V_m$ already includes the estimation error $q_m$.
Hence this argument applies directly to the remaining updates,
without restarting $x_m$ or $d_m$.
Substitute the preceding bound for $V_m$ to obtain
\begin{align*}
 V_T
 &\le\left(\frac{s}{s+n}\right)^2 e^{-2m/s}V_0
       +\frac{4\sigma^2s}{\mu(s+n)^2}
       +\frac{8\sigma^2}{\mu(s+n)}\\
 &\le e^{-2m/s}V_0+\frac{12\sigma^2}{\mu(s+n)}\\
 &\le e^{-2m/s}V_0+\frac{24\sigma^2}{\mu(T+2s)}.
\end{align*}
The second line uses $s/(s+n)\le1$ and the third uses
$n=\lceil T/2\rceil\ge T/2$.
Finally, $\Delta_T\le V_T$ and the initial bound on $V_0$
give \eqref{eq:plphase}.
\end{proof}

\section{Proof of Theorem~\ref{thm:population-nonconvex}}
\label{app:population-nonconvex}
Throughout this section,
$\eta_t=\eta$ and $a_{t+1}=1-\beta$, where
$0<\eta\le1/(2L)$ and $0<\beta<1$. The main step of the proof is to absorb the
inherited estimation error into an auxiliary iterate. The remaining
noise then has conditional mean zero and can be treated by the usual
smoothness descent argument.
\begin{proof}
\emph{The estimation error has bounded second moment.}
Let $\cF_t=\sigma(\xi_0,\ldots,\xi_{t-1})$ denote the information
available before sample $\xi_t$ is drawn.
Write $\E_t[\cdot]=\E[\cdot\mid\cF_t]$ and define
\[
 e_t=d_t-\nabla F(x_t),\qquad e_{-1}=0.
\]
For $t\ge1$, subtract $\nabla F(x_t)$ from the estimator update:
\begin{align*}
 e_t
 &=\beta e_{t-1}+\nu_t,\\
 \nu_t
 &=g(x_t,\xi_t)-\nabla F(x_t)
   -\beta\bigl(g(x_{t-1},\xi_t)-\nabla F(x_{t-1})\bigr).
\end{align*}
Set $\nu_0=e_0$. The error recursion therefore holds also at $t=0$:
\begin{equation}
 e_t=\beta e_{t-1}+\nu_t,\qquad t\ge0.
 \label{eq:population-nc-error}
\end{equation}
Both query points in $\nu_t$ are $\cF_t$-measurable. Since $\xi_t$
is independent of $\cF_t$, oracle unbiasedness gives
\begin{equation}
 \E_t\nu_t=0.
 \label{eq:population-nc-centered}
\end{equation}
The two oracle errors in $\nu_t$ use the same sample and need not
be independent. We bound their sum by the conditional $L^2$ triangle
inequality and then use \eqref{eq:bv}:
\begin{align*}
 \bigl(\E_t\norm{\nu_t}^2\bigr)^{1/2}
 &\le
 \bigl(\E_t\norm{g(x_t,\xi_t)-\nabla F(x_t)}^2\bigr)^{1/2}+\beta
 \bigl(\E_t\norm{g(x_{t-1},\xi_t)-\nabla F(x_{t-1})}^2\bigr)^{1/2}\\
 &\le(1+\beta)\sigma.
\end{align*}
For $t=0$, initialization gives the stronger bound
$\E\norm{\nu_0}^2\le\sigma^2$. Thus, for every $t\ge0$,
\begin{equation}
 \E_t\norm{\nu_t}^2\le(1+\beta)^2\sigma^2.
 \label{eq:population-nc-noise}
\end{equation}
Squaring \eqref{eq:population-nc-error} and conditioning on $\cF_t$
eliminates its cross term, because $e_{t-1}$ is $\cF_t$-measurable
and \eqref{eq:population-nc-centered} holds. Consequently,
\[
 \E\norm{e_t}^2
 \le\beta^2\E\norm{e_{t-1}}^2+(1+\beta)^2\sigma^2.
\]
Starting from $e_{-1}=0$, repeated substitution gives
\begin{equation}
 \E\norm{e_t}^2
 \le(1+\beta)^2\sigma^2\sum_{j=0}^t\beta^{2j}
 \le\frac{(1+\beta)^2}{1-\beta^2}\sigma^2
 =\frac{1+\beta}{1-\beta}\sigma^2.
 \label{eq:population-nc-error-bound}
\end{equation}

\emph{An auxiliary iterate removes the inherited error.}
Define, only for the analysis,
\begin{equation}
 z_t=x_t-\frac{\eta\beta}{1-\beta}e_{t-1}.
 \label{eq:population-nc-virtual}
\end{equation}
Then $z_0=x_0$. The factor $\beta/(1-\beta)$ is chosen so that
the error cancels from the following update:
\begin{align*}
 z_{t+1}
 &=x_{t+1}-\frac{\eta\beta}{1-\beta}e_t\\
 &=x_t-\eta\nabla F(x_t)-\eta e_t
                         -\frac{\eta\beta}{1-\beta}e_t\\
 &=x_t-\eta\nabla F(x_t)-\frac{\eta}{1-\beta}e_t\\
 &=x_t-\frac{\eta\beta}{1-\beta}e_{t-1}
          -\eta\nabla F(x_t)-\frac{\eta}{1-\beta}\nu_t.
\end{align*}
The last equality uses \eqref{eq:population-nc-error}. Hence
\begin{equation}
 z_{t+1}=z_t-\eta\nabla F(x_t)-\frac{\eta}{1-\beta}\nu_t.
 \label{eq:population-nc-virtual-update}
\end{equation}
The update has centered noise, but evaluates the gradient at $x_t$
rather than $z_t$. Standard smoothness will control this difference.
Indeed, \eqref{eq:population-nc-virtual} gives the exact identity
\[
 \norm{x_t-z_t}^2
 =\frac{\eta^2\beta^2}{(1-\beta)^2}\norm{e_{t-1}}^2.
\]
No auxiliary sequence is computed by STORM.

\emph{Descent along the auxiliary sequence.}
Apply standard smoothness to
\eqref{eq:population-nc-virtual-update}:
\begin{align*}
 F(z_{t+1})
 \le{}F(z_t)-\eta\inner{\nabla F(z_t)}
                    {\nabla F(x_t)+\nu_t/(1-\beta)}
 +\frac{L\eta^2}{2}
                 \norm{\nabla F(x_t)+\nu_t/(1-\beta)}^2.
\end{align*}
Conditional on $\cF_t$,
$\E_t\nu_t=0$. Expanding the squared norm before taking conditional
expectations therefore yields
\begin{align*}
 \E_t F(z_{t+1})
 \le{}F(z_t)-\eta\inner{\nabla F(z_t)}{\nabla F(x_t)}
                +\frac{L\eta^2}{2}\norm{\nabla F(x_t)}^2+\frac{L\eta^2}{2(1-\beta)^2}\E_t\norm{\nu_t}^2.
\end{align*}
To express the decrease in terms of the gradient at the actual iterate,
use $\inner{u}{v}\le\norm{u}^2+\norm{v}^2/4$:
\begin{align*}
 -\inner{\nabla F(z_t)}{\nabla F(x_t)}
 &=-\norm{\nabla F(x_t)}^2
   +\inner{\nabla F(x_t)-\nabla F(z_t)}{\nabla F(x_t)}\\
 &\le-\frac34\norm{\nabla F(x_t)}^2
       +\norm{\nabla F(x_t)-\nabla F(z_t)}^2\\
 &\le-\frac34\norm{\nabla F(x_t)}^2
       +\frac{L^2\eta^2\beta^2}{(1-\beta)^2}\norm{e_{t-1}}^2.
\end{align*}
Since
$L\eta\le1/2$, the coefficient of the squared gradient is at most
$-\eta(3/4-L\eta/2)\le-\eta/2$. We obtain
\begin{align*}
 \E_t F(z_{t+1})
 \le{}F(z_t)-\frac\eta2\norm{\nabla F(x_t)}^2
       +\frac{L^2\eta^3\beta^2}{(1-\beta)^2}\norm{e_{t-1}}^2+\frac{L\eta^2(1+\beta)^2}{2(1-\beta)^2}\sigma^2.
\end{align*}
Taking full expectations and using
\eqref{eq:population-nc-error-bound}, which also bounds
$\E\norm{e_{-1}}^2=0$, gives
\begin{equation}\label{eq:population-nc-descent}
 \E F(z_{t+1})
 \le{}\E F(z_t)-\frac\eta2\E\norm{\nabla F(x_t)}^2+\frac{L^2\eta^3\beta^2(1+\beta)}{(1-\beta)^3}\sigma^2
       +\frac{L\eta^2(1+\beta)^2}{2(1-\beta)^2}\sigma^2.
\end{equation}

\emph{Summation and the choice of parameters.}
Rearrange \eqref{eq:population-nc-descent} and sum over
$t=0,\ldots,T-1$. The function values telescope, so
\begin{align*}
 \frac\eta2\sum_{t=0}^{T-1}\E\norm{\nabla F(x_t)}^2
 \le{}&F(z_0)-\E F(z_T)
       +\frac{L^2\eta^3T\beta^2(1+\beta)}{(1-\beta)^3}\sigma^2+\frac{L\eta^2T(1+\beta)^2}{2(1-\beta)^2}\sigma^2.
\end{align*}
Since $z_0=x_0$ and $F(z_T)\ge F^*$, division by $\eta T/2$ gives
the more precise bound
\begin{equation}\label{eq:population-nc-precise}
 \frac1T\sum_{t=0}^{T-1}\E\norm{\nabla F(x_t)}^2
 \le{}\frac{2\Delta_0}{\eta T}
       +\frac{(1+\beta)^2}{(1-\beta)^2}L\eta\sigma^2+\frac{2\beta^2(1+\beta)}{(1-\beta)^3}L^2\eta^2\sigma^2.
\end{equation}
The condition $2L\eta\le1$ bounds the last term by
$\beta^2(1+\beta)L\eta\sigma^2/(1-\beta)^3$. Combining it with
the preceding noise term uses only the identity
\[
 \frac{(1+\beta)^2}{(1-\beta)^2}
 +\frac{\beta^2(1+\beta)}{(1-\beta)^3}
 =\frac{(1+\beta)((1+\beta)(1-\beta)+\beta^2)}{(1-\beta)^3}
 =\frac{1+\beta}{(1-\beta)^3}.
\]
This proves \eqref{eq:population-nc-general}.

For $\eta=1/(2L\sqrt T)$, substitution gives
\[
 \frac1T\sum_{t=0}^{T-1}\E\norm{\nabla F(x_t)}^2
 \le\frac{4L\Delta_0+(1+\beta)\sigma^2/(2(1-\beta)^3)}{\sqrt T}.
\]
If $0<\beta\le1/2$, then
$(1+\beta)/(2(1-\beta)^3)\le(3/2)/(2(1/2)^3)=6$,
which proves the uniform bound.
For an independent uniform index $I\in\{0,\ldots,T-1\}$,
\[
 \E\norm{\nabla F(x_I)}
 \le\sqrt{\E\norm{\nabla F(x_I)}^2}
 =\sqrt{\frac1T\sum_{t=0}^{T-1}\E\norm{\nabla F(x_t)}^2}.
\]
\end{proof}

\emph{Dependence on the variance.}
For $0<\beta\le1/2$, \eqref{eq:population-nc-general} also gives
\[
 \frac1T\sum_{t=0}^{T-1}\E\norm{\nabla F(x_t)}^2
 \le\frac{2\Delta_0}{\eta T}+12L\eta\sigma^2.
\]
When $\Delta_0,\sigma>0$, choose
$\eta=\min\{1/(2L),\sqrt{\Delta_0/(6L\sigma^2T)}\}$.
The elementary inequalities
\[
 \eta^{-1}\le2L+\sqrt{\frac{6L\sigma^2T}{\Delta_0}},
 \qquad
 \eta\le\sqrt{\frac{\Delta_0}{6L\sigma^2T}}
\]
then yield
\[
 \frac1T\sum_{t=0}^{T-1}\E\norm{\nabla F(x_t)}^2
 \le\frac{4L\Delta_0}{T}
       +4\sqrt6\,\sigma\sqrt{\frac{L\Delta_0}{T}}.
\]
Thus the bound recovers the usual initialization and variance orders
of nonconvex SGD. 

If $\sigma=0$, taking $\eta=1/(2L)$ directly
gives $4L\Delta_0/T$. 

\section{Proof of Lemma~\ref{lem:geometry-main}}\label{app:companion}
\begin{proof}
For fixed $y$, let $\phi(z)=F(z)-\inner{\nabla F(y)}{z}$.
Convexity and $\nabla\phi(y)=0$ imply that $y$ minimizes $\phi$.
The smoothness upper model at $x$, evaluated at
$x-L^{-1}\nabla\phi(x)$, therefore gives
\[
 \phi(y)\le\phi\bigl(x-L^{-1}\nabla\phi(x)\bigr)
 \le\phi(x)-\frac{\norm{\nabla\phi(x)}^2}{2L}.
\]
Substituting the definition of $\phi$ yields
\[
 F(x)-F(y)-\inner{\nabla F(y)}{x-y}
 \ge\frac{\norm{\nabla F(x)-\nabla F(y)}^2}{2L}.
\]
Interchanging $x,y$ and adding cancels the function values. With
$u=x-y$ and $v=\nabla F(x)-\nabla F(y)$, the result is
$\inner{v}{u}\ge\norm{v}^2/L$. Hence
\[
 \norm{u-\eta v}^2
 =\norm{u}^2-2\eta\inner{u}{v}+\eta^2\norm{v}^2
 \le\norm{u}^2-\eta(2/L-\eta)\norm{v}^2
 \le\norm{u}^2.
\]
Taking square roots proves \eqref{eq:nonexpansive-main}.

\emph{Contraction under strong convexity.}
Assume first that $L>\lambda$ and set
$H(z)=F(z)-\lambda\norm{z}^2/2$. Subtracting the quadratic term
from the strong-convexity lower model and the smoothness upper model gives
\[
 0\le H(z)-H(w)-\inner{\nabla H(w)}{z-w}
       \le\frac{L-\lambda}{2}\norm{z-w}^2.
\]
Thus $H$ is convex and has an upper quadratic model with coefficient
$L-\lambda$. The preceding proof of nonexpansiveness uses only these
two properties, so it applies to $H$. Since $\eta\le1/L$ and
$L>\lambda$, we have $1-\lambda\eta>0$ and
\[
 \frac{\eta}{1-\lambda\eta}\le\frac1{L-\lambda};
\]
after multiplying by the positive denominators, this is just
$L\eta\le1$. Using $\nabla F(z)=\nabla H(z)+\lambda z$,
\begin{align*}
 &\norm{x-\eta\nabla F(x)-y+\eta\nabla F(y)}\\
 &\quad=(1-\lambda\eta)
 \norm{x-\frac{\eta}{1-\lambda\eta}\nabla H(x)
          -y+\frac{\eta}{1-\lambda\eta}\nabla H(y)}\\
 &\quad\le(1-\lambda\eta)\norm{x-y}.
\end{align*}
If $L=\lambda$, the lower and upper models for $H$ coincide, so $H$
is affine. Hence $\nabla F(x)-\nabla F(y)=\lambda(x-y)$ and
\eqref{eq:contraction-main0} holds with equality.

\emph{Comparison of function gaps.}
The smoothness upper model evaluated at one gradient step gives
\[
 F^*\le F\bigl(y-L^{-1}\nabla F(y)\bigr)
 \le F(y)-\frac{\norm{\nabla F(y)}^2}{2L}.
\]
Rearrangement proves \eqref{eq:self-bound-main}. Apply the upper model
once more to $F(x)$, and use Young's inequality:
\begin{align*}
 F(x)-F^*
 &\le F(y)-F^*+\inner{\nabla F(y)}{x-y}+\frac L2\norm{x-y}^2\\
 &\le F(y)-F^*+\frac{\norm{\nabla F(y)}^2}{2L}+L\norm{x-y}^2\\
 &\le2(F(y)-F^*)+L\norm{x-y}^2.
\end{align*}
The two function-gap inequalities require smoothness only.
\end{proof}

\section{Proof of Lemma~\ref{lem:companion}}
\begin{proof}
Let $\mathcal F_t=\sigma(\xi_0,\ldots,\xi_{t-1})$ be the information
before drawing $\xi_t$.
Write $\zeta(x,\xi)=g(x,\xi)-\nabla F(x)$ and
$e_t=d_t-\nabla F(x_t)$. Set $e_{-1}=0$, $\nu_0=e_0$, and
$\nu_t=\zeta(x_t,\xi_t)-\beta\zeta(x_{t-1},\xi_t)$ for $t\ge1$.
Subtracting $\nabla F(x_t)$ from \eqref{eq:estimator} yields
$e_t=\beta e_{t-1}+\nu_t$. Fresh-sample independence and unbiasedness give
$\E[\nu_t\mid\mathcal F_t]=0$. The conditional $L^2$
triangle inequality gives
\begin{align*}
 \bigl(\E[\norm{\nu_t}^2\mid\mathcal F_t]\bigr)^{1/2}
\le\bigl(\E[\norm{\zeta(x_t,\xi_t)}^2\mid\mathcal F_t]\bigr)^{1/2}+\beta\bigl(\E[\norm{\zeta(x_{t-1},\xi_t)}^2
                                  \mid\mathcal F_t]\bigr)^{1/2}
 \le(1+\beta)\sigma.
\end{align*}
At $t=0$ the same bound follows directly from \eqref{eq:bv}.
Thus, with $S^2=(1+\beta)^2\sigma^2$,
\begin{equation}
 \E[\nu_t\mid\mathcal F_t]=0,\qquad
 \E[\norm{\nu_t}^2\mid\mathcal F_t]\le S^2\le9\sigma^2/4.
 \label{eq:innovation-app}
\end{equation}

\emph{The estimation error.}
The previous error is $\mathcal F_t$-measurable, so
$\E\inner{e_{t-1}}{\nu_t}=0$. Squaring the error recursion and
substituting recursively from $e_{-1}=0$ gives
\[
 \E\norm{e_t}^2\le \beta^2\E\norm{e_{t-1}}^2+S^2
 \le S^2\sum_{j=0}^t \beta^{2j}
 \le\frac{(1+\beta)^2\sigma^2}{1-\beta^2}
 =\frac{1+\beta}{1-\beta}\sigma^2\le3\sigma^2.
\]
Define the mathematical companion by
\begin{equation}
 y_0=x_0,\qquad y_{t+1}=y_t-\eta_t(\nabla F(y_t)+\nu_t).
 \label{eq:companion-app}
\end{equation}
Subtracting \eqref{eq:companion-app} from \eqref{eq:primal} cancels
the shared innovation:
\[
 x_{t+1}-y_{t+1}
 =x_t-\eta_t\nabla F(x_t)-y_t+\eta_t\nabla F(y_t)
                                      -\eta_t \beta e_{t-1}.
\]
For $r_t=(\E\norm{x_t-y_t}^2)^{1/2}$, Minkowski's inequality,
\eqref{eq:nonexpansive-main}, and the estimation-error bound above yield
$r_{t+1}\le r_t+\sqrt3 \beta\sigma\eta_t$.
Summing $r_{j+1}-r_j\le\sqrt3\beta\sigma\eta_j$ from $j=0$
to $t-1$, with $r_0=0$, gives
$r_t\le\sqrt3\beta\sigma\sum_{j<t}\eta_j$. Squaring gives
\begin{equation}
 \E\norm{x_t-y_t}^2
 \le3\beta^2\sigma^2\left(\sum_{j=0}^{t-1}\eta_j\right)^2.
 \label{eq:coupling-convex-app}
\end{equation}
Under $\lambda$-strong convexity, \eqref{eq:contraction-main0} instead gives
$r_{t+1}\le(1-\lambda\eta_t)r_t+\sqrt3 \beta\sigma\eta_t$.
Starting from $r_0=0$, the induction
\[
 r_{t+1}\le(1-\lambda\eta_t)\frac{\sqrt3 \beta\sigma}{\lambda}
                         +\sqrt3 \beta\sigma\eta_t
              =\frac{\sqrt3 \beta\sigma}{\lambda}
\]
proves
\begin{equation}
 \E\norm{x_t-y_t}^2\le3\beta^2\sigma^2/\lambda^2.
 \label{eq:coupling-strong-app}
\end{equation}
Under strong convexity, this gives
\begin{equation}
 \E[F(x_t)-F^*]
 \le2\E[F(y_t)-F^*]+3L\beta^2\sigma^2/\lambda^2.
 \label{eq:expected-transfer-app}
\end{equation}
Combining the same pointwise inequality with
\eqref{eq:coupling-convex-app} gives the convex assertion explicitly:
\begin{equation}
 \E[F(x_t)-F^*]\le2\E[F(y_t)-F^*]
       +3L\beta^2\sigma^2\left(\sum_{j=0}^{t-1}\eta_j\right)^2.
 \label{eq:convex-transfer-app}
\end{equation}
\end{proof}

\section{Other useful lemmas}
\subsection{A shared one-step inequality}
\begin{lemma}
\label{lem:convex-distance}
Let $y_t$ be an adapted recursion of the form \eqref{eq:companion-app}
with \eqref{eq:innovation-app}, where $y_t$ is $\mathcal F_t$-measurable.
Under Assumptions~\ref{ass:smooth} and~\ref{ass:convex}, suppose
$y_0=x_0$. Define
$D_t=\E\norm{y_t-x^*}^2$, $\delta_t=\E[F(y_t)-F^*]$,
and $h_t=\nabla F(y_t)+\nu_t$. For deterministic
$0<\eta_t\le1/(2L)$,
\begin{align}
 \E\norm{h_t}^2&\le2L\delta_t+S^2,
 \label{eq:convex-moment}\\
 D_{t+1}&\le D_t-\eta_t\delta_t+\eta_t^2S^2.
 \label{eq:convex-distance}
\end{align}
\end{lemma}
\begin{proof}
Note that
$\norm{\nabla F(y_t)}\le L\norm{y_t-x^*}$ and conditional centering justifies
$\E\inner{\nabla F(y_t)}{\nu_t}=0$. Expanding $\norm{h_t}^2$
and using \eqref{eq:self-bound-main} proves \eqref{eq:convex-moment}.
For distance, the same centering and convexity give
\begin{align*}
 D_{t+1}
 &=D_t-2\eta_t\E\inner{h_t}{y_t-x^*}+\eta_t^2\E\norm{h_t}^2\\
 &\le D_t-2\eta_t\E\inner{\nabla F(y_t)}{y_t-x^*}
                                      +\eta_t^2(2L\delta_t+S^2)\\
 &\le D_t-2\eta_t(1-L\eta_t)\delta_t+\eta_t^2S^2\\
 &\le D_t-\eta_t\delta_t+\eta_t^2S^2.
\end{align*}
The third line uses
$\inner{\nabla F(y_t)}{y_t-x^*}\ge F(y_t)-F^*$.
The last uses $2(1-L\eta_t)\ge1$ and $\delta_t\ge0$.
\end{proof}

\subsection{Rounded-block geometry for arbitrary integer horizons}
\begin{lemma}
\label{lem:rounded-geometry}
For the partition \eqref{eq:rounded-boundaries} and $T\ge1$, put
$M=\lceil (T+1)/4\rceil$, $K=\lceil (T+1)/2\rceil$. Define
\begin{equation}
 B_0=\{M-1,\ldots,K-1\},\qquad
 B_i=\{U_i,\ldots,U_{i+1}-1\}\quad(1\le i\le m).
 \label{eq:comparison-blocks}
\end{equation}
Then $U_1=K$, $B_m=\{T\}$, and
\begin{gather}
 M\ge (T+1)/4,\quad K\ge (T+1)/2,\quad K\in\{2M-1,2M\},\nonumber\\
 n:=K-M+1\ge (T+1)/4,\quad K/n\le2,
 \label{eq:rounded-geometry}\\
 |B_0|\le2|B_1|,\qquad
 |B_i|\le3|B_{i+1}|\quad(1\le i<m).
 \label{eq:rounded-block-ratios}
\end{gather}
In particular, a suffix of $B_i$ together with all of $B_{i+1}$
contains at most $4|B_{i+1}|$ indices.
\end{lemma}
\begin{proof}
The identities $U_1=\lceil (T+1)/2\rceil$ and
$\lfloor (T+1)/2^m\rfloor=1$ prove the boundary assertions.
Writing $T+1=4q+r$ with $r\in\{0,1,2,3\}$ verifies
$K\in\{2M-1,2M\}$. Hence $n\ge M\ge (T+1)/4$ and $K\le2n$.
The same four cases give
\[
 \begin{array}{c|cccc}
 r&0&1&2&3\\ \hline
 |B_0|&q+1&q+1&q+1&q+2\\
 |B_1|&q&q&q+1&q+1.
 \end{array}
\]
For $r=0,1$, $T\ge1$ implies $q\ge1$; for $r=2,3$, $q\ge0$
is enough. Thus $|B_0|\le2|B_1|$ in every case, including
$T=1,2$.

For the remaining ratios, let $r_i=\lfloor (T+1)/2^i\rfloor$.
For $1\le i<m$, $r_{i+1}=\lfloor r_i/2\rfloor$ and
$|B_i|=r_i-r_{i+1}=\lceil r_i/2\rceil$. For $1\le i<m-1$,
write $r_i=2r_{i+1}+\varepsilon$ with $\varepsilon\in\{0,1\}$.
It follows that
\[
 |B_i|=r_{i+1}+\varepsilon
 \le2\lceil r_{i+1}/2\rceil+1
 =2|B_{i+1}|+1\le3|B_{i+1}|.
\]
For the final pair, $r_m=1$ and $r_{m-1}\in\{2,3\}$,
so $|B_{m-1}|\le2$ and $|B_m|=1$. 
\end{proof}

\section{Proof of Theorem~\ref{thm:convex}}
\label{app:convex}

\subsection{General stepsize-weighted and constant-step means}
\begin{theorem}
\label{thm:convex-mean}
Under Assumptions~\ref{ass:smooth}, \ref{ass:oracle}, and
\ref{ass:convex}, let $T\ge1$
be any integer, take $\beta=1/(T+2)$, and choose deterministic stepsizes
$0<\eta_t\le1/(2L)$ for $0\le t\le T$. Define
\[
 \Lambda_T=\sum_{t=0}^T\eta_t,\qquad
 \bar x_T^\eta=\frac1{\Lambda_T}\sum_{t=0}^T\eta_t x_t.
\]
Then
\begin{equation}
 \E[F(\bar x_T^\eta)-F^*]
 \le\frac{2R^2}{\Lambda_T}
     +\frac{9\sigma^2}{2\Lambda_T}\sum_{t=0}^T\eta_t^2
     +\frac{3L\sigma^2\Lambda_T^2}{(T+2)^2}.
 \label{eq:convex-weighted-general}
\end{equation}
\end{theorem}
\begin{proof}
Apply Lemma~\ref{lem:companion}, and use the notation
$D_t=\E\norm{y_t-x^*}^2$ and
$\delta_t=\E[F(y_t)-F^*]$.
Rearrange \eqref{eq:convex-distance} as
$\eta_t\delta_t\le D_t-D_{t+1}+\eta_t^2S^2$.
Summing over $t=0,\ldots,T$ cancels each interior distance:
\begin{align*}
 \sum_{t=0}^T\eta_t\delta_t
 &\le\sum_{t=0}^T(D_t-D_{t+1})+S^2\sum_{t=0}^T\eta_t^2=D_0-D_{T+1}+S^2\sum_{t=0}^T\eta_t^2
 \le R^2+S^2\sum_{t=0}^T\eta_t^2.
\end{align*}
The final distance is nonnegative and can be discarded.

The weights $\eta_t/\Lambda_T$ are positive and sum to one.
Apply convexity to the weighted mean of STORM iterates, then apply
\eqref{eq:convex-transfer-app} to each summand:
\begin{align*}
 \E[F(\bar x_T^\eta)-F^*]
 &\le\frac1{\Lambda_T}\sum_{t=0}^T\eta_t
                                  \E[F(x_t)-F^*]\\
 &\le\frac2{\Lambda_T}\sum_{t=0}^T\eta_t\delta_t
    +\frac{3L\beta^2\sigma^2}{\Lambda_T}
        \sum_{t=0}^T\eta_t\left(\sum_{j=0}^{t-1}\eta_j\right)^2\\
 &\le\frac{2R^2}{\Lambda_T}
       +\frac{2S^2}{\Lambda_T}\sum_{t=0}^T\eta_t^2
       +3L\beta^2\sigma^2\Lambda_T^2.
\end{align*}
For the last term, every prefix sum is at most $\Lambda_T$, and
$\sum_t\eta_t=\Lambda_T$. Substituting $\beta=1/(T+2)$ and
$S^2\le9\sigma^2/4$ proves the theorem.
\end{proof}

\begin{corollary}
\label{cor:convex-constant-app}
With $\eta_t=\eta\le1/(2L)$, the weighted mean becomes the uniform
mean such that $\bar x_T=(T+1)^{-1}\sum_{t=0}^T x_t$, and
\begin{equation}
 \E[F(\bar x_T)-F^*]
 \le\frac{2R^2}{\eta (T+1)}+6\eta\sigma^2.
 \label{eq:convex-constant}
\end{equation}
If $R>0$ and $\sigma>0$, choose
$\eta=\min\{1/(2L),\,R/(\sqrt3\sigma\sqrt{T+1})\}$ to obtain
\begin{equation}
 \E[F(\bar x_T)-F^*]
 \le\frac{4LR^2}{T+1}+\frac{4\sqrt3\sigma R}{\sqrt{T+1}}.
 \label{eq:convex-constant-tuned}
\end{equation}
\end{corollary}

\begin{proof}
Here $\Lambda_T=(T+1)\eta$ and $\sum_t\eta_t^2=(T+1)\eta^2$.
Since $(T+1)/(T+2)\le1$ and $L\eta\le1/2$, the last two terms of
\eqref{eq:convex-weighted-general} are at most
$(9/2)\eta\sigma^2+3L\eta^2\sigma^2\le6\eta\sigma^2$.
For the tuned result, use
\[
 \frac1\eta=\max\{2L,\sqrt3\sigma\sqrt{T+1}/R\}
        \le2L+\sqrt3\sigma\sqrt{T+1}/R,\qquad
 \eta\le\frac R{\sqrt3\sigma\sqrt{T+1}}.
\]
The first inequality bounds the initialization term by
$4LR^2/(T+1)+2\sqrt3\sigma R/\sqrt{T+1}$; the second bounds the
noise term by $2\sqrt3\sigma R/\sqrt{T+1}$, giving the stated bound.
If $\sigma=0$, take $\eta=1/(2L)$.
\end{proof}

\subsection{The companion's last-iterate bound}
\begin{lemma}
\label{lem:convex-blocks}
Under Assumptions~\ref{ass:smooth} and~\ref{ass:convex}, let
$y_{t+1}=y_t-\eta_t(\nabla F(y_t)+\nu_t)$, with $y_0=x_0$,
$y_t$ measurable with respect to the filtration $\mathcal F_t$,
$\E[\nu_t\mid\mathcal F_t]=0$, and
$\E[\norm{\nu_t}^2\mid\mathcal F_t]\le S^2$.
For $T\ge1$ and the stepsizes in \eqref{eq:convex-schedule},
\begin{equation}
 \max\left\{\delta_T,\,
                  \frac1{T+1}\sum_{t=0}^T\delta_t\right\}
 \le\frac{4\mathrm e R^2}{\eta (T+1)}+18\mathrm e\eta S^2,
 \qquad \delta_t=\E[F(y_t)-F^*].
 \label{eq:convex-companion-last}
\end{equation}
\end{lemma}

\begin{proof}
Use $D_t=\E\norm{y_t-x^*}^2$ and
$h_t=\nabla F(y_t)+\nu_t$. Note that \eqref{eq:convex-moment}--\eqref{eq:convex-distance}
apply. Put $M=\lceil (T+1)/4\rceil$ and $K=\lceil (T+1)/2\rceil$.
The first schedule block is $0,\ldots,K-1$, with
$\eta_t=\eta$.

\emph{A reference value in the first half.}
Discard the nonpositive term $-\eta\delta_t$ in
\eqref{eq:convex-distance} for $0\le t<M-1$ and sum. This yields
\[
 D_{M-1}\le D_0+(M-1)\eta^2S^2
                   =R^2+(M-1)\eta^2S^2.
\]
Now keep that decrease and sum from $t=M-1$ to $K-1$:
\begin{align*}
 \eta\sum_{t=M-1}^{K-1}\delta_t
 &\le D_{M-1}-D_K+(K-M+1)\eta^2S^2\\
 &\le R^2+\bigl((M-1)+(K-M+1)\bigr)\eta^2S^2\\
 &=R^2+K\eta^2S^2.
\end{align*}
Let $A_0$ be the minimum of $\delta_t$ over these $K-M+1$
indices. A minimum is at most the mean, so
\begin{equation}
 A_0\le\frac{R^2}{(K-M+1)\eta}
                   +\frac{K}{K-M+1}\eta S^2
       \le\frac{4R^2}{\eta (T+1)}+2\eta S^2.
 \label{eq:convex-first-reference}
\end{equation}
The last line uses $K-M+1\ge (T+1)/4$ and
$K/(K-M+1)\le2$ from Lemma~\ref{lem:rounded-geometry}.
We will also need the mean over the entire first block.
Summing directly from $0$ to $K-1$ gives
\begin{equation}
 \frac1K\sum_{t=0}^{K-1}\delta_t
 \le\frac{R^2}{K\eta}+\eta S^2
 \le\frac{2R^2}{\eta (T+1)}+\eta S^2,
 \label{eq:convex-first-mean}
\end{equation}
where $K\ge (T+1)/2$.

\emph{Comparison with an earlier point.}
Fix deterministic indices $p\le q$. For every $p\le t\le q$,
expand the squared distance from $y_p$:
\[
 \norm{y_{t+1}-y_p}^2
 =\norm{y_t-y_p}^2-2\eta_t\inner{h_t}{y_t-y_p}
                                      +\eta_t^2\norm{h_t}^2.
\]
Although $y_p$ is random, it is $\mathcal F_t$-measurable when
$t\ge p$. Conditional centering therefore gives
\begin{align*}
 \E\inner{h_t}{y_t-y_p}
 &=\E\inner{\nabla F(y_t)}{y_t-y_p}\\
 &\ge\E[F(y_t)-F(y_p)]=\delta_t-\delta_p.
\end{align*}
The inequality follows by rearranging
$F(y_p)\ge F(y_t)+\inner{\nabla F(y_t)}{y_p-y_t}$.
After taking expectations and summing the distance identity,
the initial distance is zero and the terminal distance is nonnegative:
\begin{align*}
 2\sum_{t=p}^q\eta_t(\delta_t-\delta_p)
 &\le-\E\norm{y_{q+1}-y_p}^2
                               +\sum_{t=p}^q\eta_t^2\E\norm{h_t}^2\\
 &\le2L\sum_{t=p}^q\eta_t^2\delta_t+S^2\sum_{t=p}^q\eta_t^2.
\end{align*}
The right side still depends on the function errors. Write
$\delta_t=(\delta_t-\delta_p)+\delta_p$ there, and move the
first part to the left. This is an exact rearrangement, giving
\begin{equation}
 2\sum_{t=p}^q\eta_t(1-L\eta_t)(\delta_t-\delta_p)
 \le(2L\delta_p+S^2)\sum_{t=p}^q\eta_t^2.
 \label{eq:convex-reference}
\end{equation}
The differences $\delta_t-\delta_p$ need not be nonnegative.

\emph{Comparison of consecutive blocks.}
Use the comparison blocks \eqref{eq:comparison-blocks}.
For each $i$, choose a deterministic minimizer $\tau_i$ of
$\delta_t$ on $B_i$ and set $A_i=\delta_{\tau_i}$.
These are indices minimizing expected errors, not data-dependent
choices made by the algorithm. Write
$\bar\delta_i=|B_i|^{-1}\sum_{t\in B_i}\delta_t$ for a block mean.
In particular, $A_i\le\bar\delta_i$ and
$A_m=\bar\delta_m=\delta_T$ because $B_m=\{T\}$.

In \eqref{eq:convex-reference}, take $p=\tau_i$ and
$q=\max B_{i+1}$ for $0\le i<m$. On the remaining part of
$B_i$, each difference $\delta_t-A_i$ is nonnegative.
Its coefficient $2\eta_t(1-L\eta_t)$ is positive, so discarding
this source part can only decrease the left side.
Throughout the target block $B_{i+1}$ the stepsize equals
$\eta2^{-(i+1)}$. The remaining inequality is therefore
\begin{align*}
 &2\eta2^{-(i+1)}(1-L\eta2^{-(i+1)})|B_{i+1}|
                              (\bar\delta_{i+1}-A_i)\le(2LA_i+S^2)
                      \sum_{t=\tau_i}^{\max B_{i+1}}\eta_t^2.
\end{align*}
Every comparison span has at most $4|B_{i+1}|$ indices by
\eqref{eq:rounded-block-ratios}. Its largest stepsize is
at most $\eta2^{-i}$. Moreover,
$1-L\eta2^{-(i+1)}\ge1/2$. Thus the ratio of the sum of squared
stepsizes to the coefficient of $\bar\delta_{i+1}-A_i$ is at most
\[
 \frac{4|B_{i+1}|\eta^2 2^{-2i}}
      {|B_{i+1}|\eta2^{-(i+1)}}
 =8\eta2^{-i}=:c_i.
\]
Since $2LA_i+S^2\ge0$, it follows that
\begin{equation}
 A_{i+1}\le\bar\delta_{i+1}
            \le(1+2Lc_i)A_i+c_iS^2,
 \qquad c_i=8\eta2^{-i}.
 \label{eq:convex-block-recursion}
\end{equation}
The mean difference $\bar\delta_{i+1}-A_i$ itself may be negative;
division by its positive coefficient remains valid in that case.
The length bounds used here include $T=1$ and $T=2$, whose blocks
are respectively $B_0=\{0\}, B_1=\{1\}$ and
$B_0=\{0,1\}, B_1=\{2\}$.

\emph{Expansion of the scalar recursion.}
The first two substitutions in \eqref{eq:convex-block-recursion}
(the second when $m\ge2$) are
\begin{align*}
 A_1&\le(1+2Lc_0)A_0+c_0S^2,\\
 A_2&\le(1+2Lc_1)(1+2Lc_0)A_0
                  +\bigl(c_0(1+2Lc_1)+c_1\bigr)S^2.
\end{align*}
More generally, induction gives, for $0\le i<m$,
\begin{equation}
 \bar\delta_{i+1}
 \le A_0\prod_{j=0}^{i}(1+2Lc_j)
       +S^2\sum_{j=0}^{i}c_j
                    \prod_{\ell=j+1}^{i}(1+2Lc_\ell).
 \label{eq:convex-block-expanded}
\end{equation}
To see the induction step, use the same formula with index $i-1$
to bound $A_i\le\bar\delta_i$, multiply it by $1+2Lc_i$, and
add $c_iS^2$. This extends every existing product by its last
factor and adds the $j=i$ summand, whose empty product equals one.

The coefficients have a geometric sum:
\[
 \sum_{j=0}^{m-1}c_j=8\eta\sum_{j=0}^{m-1}2^{-j}\le16\eta.
\]
For any product occurring in \eqref{eq:convex-block-expanded}, use
$1+u\le\exp(u)$ for $u\ge0$ to obtain
\[
 \prod(1+2Lc_j)
 \le\exp\left(2L\sum_{j=0}^{m-1}c_j\right)
 \le\exp(32L\eta)\le\mathrm e.
\]
Hence every tail-block mean satisfies
\begin{align}
 \bar\delta_{i+1}
 &\le\mathrm e A_0+16\mathrm e\eta S^2\nonumber\\
 &\le\frac{4\mathrm e R^2}{\eta (T+1)}
                              +18\mathrm e\eta S^2.
 \label{eq:convex-each-block}
\end{align}
The second line inserts \eqref{eq:convex-first-reference}.
The coefficients $c_j$ form a geometric sequence, so neither term
incurs a factor proportional to the number of blocks.

\emph{The last iterate and the uniform mean.}
The last block is a singleton, so \eqref{eq:convex-each-block}
with $i+1=m$ is already the desired last-point inequality.
For the full mean, partition $0,\ldots,T$ into the entire
first schedule block and $B_1,\ldots,B_m$:
\[
 \frac1{T+1}\sum_{t=0}^T\delta_t
 =\frac K{T+1}\left(\frac1K\sum_{t=0}^{K-1}\delta_t\right)
       +\sum_{i=1}^m\frac{|B_i|}{T+1}\bar\delta_i.
\]
The coefficients on the right sum to one. The first mean is
bounded by \eqref{eq:convex-first-mean}, which is no larger than
the right side of \eqref{eq:convex-each-block}; every remaining
mean is bounded by \eqref{eq:convex-each-block} itself.
Their convex combination obeys the same bound. This completes
the proof of \eqref{eq:convex-companion-last}.
\end{proof}

\subsection{Proof of the convex theorem}
\begin{proof}[Proof of Theorem~\ref{thm:convex}]
Lemma~\ref{lem:companion} gives
$S^2\le9\sigma^2/4$ and a companion to which
Lemma~\ref{lem:convex-blocks} applies.
Because $\eta_t\le\eta$ and $\beta=1/(T+2)$, for every $0\le t\le T$,
\begin{equation}
 \E\norm{x_t-y_t}^2
 \le3\beta^2\sigma^2(t\eta)^2\le3\sigma^2\eta^2.
 \label{eq:convex-uniform-coupling}
\end{equation}

Transfer the companion's last-point guarantee using
\eqref{eq:convex-transfer-app} and \eqref{eq:convex-uniform-coupling}:
\begin{align*}
 \E[F(x_T)-F^*]
 &\le2\delta_T+3L\sigma^2\eta^2\\
 &\le\frac{8\mathrm eR^2}{\eta (T+1)}
                       +36\mathrm e\eta S^2+3L\sigma^2\eta^2\\
 &\le\frac{8\mathrm eR^2}{\eta (T+1)}
                       +(81\mathrm e+3/32)\eta\sigma^2\\
 &\le\frac{8\mathrm eR^2}{\eta (T+1)}+82\mathrm e\eta\sigma^2.
\end{align*}
Here $36(9/4)=81$ and $L\eta\le1/32$.
For the uniform mean, convexity gives
\[
 \E[F(\bar x_T)-F^*]
 \le\frac1{T+1}\sum_{t=0}^T\E[F(x_t)-F^*]
 \le\frac2{T+1}\sum_{t=0}^T\delta_t+3L\sigma^2\eta^2.
\]
Use the mean part of \eqref{eq:convex-companion-last} and the
same numerical calculation. This proves the last-point and uniform-mean bound.

\emph{The stepsize-weighted mean.}
Use Theorem~\ref{thm:convex-mean} directly, and put $K=\lceil (T+1)/2\rceil$.
The first $K$ weights
equal $\eta$, every weight is at most $\eta$, and $K\ge (T+1)/2$.
Therefore
\[
 \Lambda_T\ge K\eta\ge (T+1)\eta/2,\qquad
 \sum_{t=0}^T\eta_t^2\le\eta\Lambda_T,\qquad
 \frac{\Lambda_T}{T+2}\le\eta.
\]
Insert these three inequalities into
\eqref{eq:convex-weighted-general}:
\begin{align*}
 \E[F(\bar x_T^\eta)-F^*]
 &\le\frac{4R^2}{\eta (T+1)}+\frac92\eta\sigma^2
                                      +3L\eta^2\sigma^2\\
 &\le\frac{4R^2}{\eta (T+1)}
                          +\left(\frac92+\frac3{32}\right)\eta\sigma^2
 \le\frac{4R^2}{\eta (T+1)}+5\eta\sigma^2.
\end{align*}
This gives the weighted-mean bound. Its right side is no larger than
the common bound for the last iterate and uniform mean, so that
bound applies to all three outputs.

\emph{Choice of the base stepsize.}
For the stepsize choice in Theorem~\ref{thm:convex},
\[
 \frac1\eta
 =\max\left\{32L,\frac{\sqrt{41}\sigma\sqrt{T+1}}{2R}\right\}
 \le32L+\frac{\sqrt{41}\sigma\sqrt{T+1}}{2R},\qquad
 \eta\le\frac{2R}{\sqrt{41}\sigma\sqrt{T+1}}.
\]
Thus the two terms of the common bound satisfy
\begin{align*}
 \frac{8\mathrm eR^2}{\eta (T+1)}
 &\le\frac{256\mathrm eLR^2}{T+1}
                      +\frac{4\sqrt{41}\,\mathrm e\sigma R}{\sqrt{T+1}},\\
 82\mathrm e\eta\sigma^2
 &\le\frac{164\mathrm e\sigma R}{\sqrt{41}\sqrt{T+1}}
       =\frac{4\sqrt{41}\,\mathrm e\sigma R}{\sqrt{T+1}}.
\end{align*}
Adding proves the tuned bound.
If $\sigma=0$, choose $\eta=1/(32L)$ and obtain the first term
of the tuned bound. If $R=0$, the separate zero-step
choice returning $x_0=x^*$ gives zero error.
\end{proof}

\section{Proof of Theorems~\ref{thm:strong-mean-main} and~\ref{thm:strong}}
\label{app:strongly-convex}
\subsection{The linearly weighted mean}
\begin{proof}[Proof of Theorem~\ref{thm:strong-mean-main}]
Use the parameters in \eqref{eq:sc-mean-schedule}, and write
$W_T=\sum_{t=0}^T(t+s)=(T+1)(T+2s)/2$.
The choices satisfy $\beta\le1/2$ and $\eta_t\le1/(2L)$, so the companion
lemma applies. For the remainder of this proof write
\[
 D_t=\E\norm{y_t-x^*}^2,\qquad
 \delta_t=\E[F(y_t)-F^*].
\]

\emph{The one-step inequality.}
Expand the squared distance in \eqref{eq:companion-app}. The cross terms
containing $\nu_t$ vanish after conditioning on $\mathcal F_t$, giving
\begin{align*}
 D_{t+1}
 &\le D_t-2\eta_t\E\inner{\nabla F(y_t)}{y_t-x^*}
       +\eta_t^2\E\norm{\nabla F(y_t)}^2+\eta_t^2S^2\\
 &\le(1-\lambda\eta_t)D_t
       -2\eta_t(1-L\eta_t)\delta_t+\eta_t^2S^2\\
 &\le(1-\lambda\eta_t)D_t-\eta_t\delta_t+\eta_t^2S^2.
\end{align*}
The second line uses
$\inner{\nabla F(y_t)}{y_t-x^*}\ge F(y_t)-F^*
+\lambda\norm{y_t-x^*}^2/2$, together with \eqref{eq:self-bound-main}. The third line
uses $L\eta_t\le1/2$. Rearranging and substituting the stepsize gives
\begin{equation}
 \delta_t\le\frac\lambda2(t+s-1)D_t
            -\frac\lambda2(t+s+1)D_{t+1}
            +\frac{2S^2}{\lambda(t+s+1)}.
 \label{eq:average-one-step}
\end{equation}

\emph{Summation with increasing weights.}
Multiply \eqref{eq:average-one-step} by $t+s$:
\begin{align*}
 (t+s)\delta_t
 \le\frac\lambda2\bigl((t+s)(t+s-1)D_t
                  -(t+s)(t+s+1)D_{t+1}\bigr)+\frac{2S^2}{\lambda}\frac{t+s}{t+s+1}.
\end{align*}
At $t=0$, the distance terms are
$s(s-1)D_0-s(s+1)D_1$. At $t=1$, they are
$s(s+1)D_1-(s+1)(s+2)D_2$. Thus the $D_1$ terms cancel,
and the same cancellation occurs at every interior index. Summing gives
\begin{align*}
 \sum_{t=0}^T(t+s)\delta_t
 &\le\frac\lambda2\bigl(s(s-1)D_0
                 -(T+s)(T+s+1)D_{T+1}\bigr)
       +\frac{2S^2}{\lambda}\sum_{t=0}^T\frac{t+s}{t+s+1}\\
 &\le\frac{\lambda s(s-1)R^2}{2}+\frac{2(T+1)S^2}{\lambda}.
\end{align*}
Here $D_{T+1}\ge0$, every fraction is at most one, and $D_0=R^2$.

\emph{Transfer to the STORM iterates.}
Convexity and \eqref{eq:expected-transfer-app} imply
\begin{align*}
 \E[F(\bar x_T)-F^*]
 &\le\frac1{W_T}\sum_{t=0}^T(t+s)\E[F(x_t)-F^*]\\
 &\le\frac2{W_T}\sum_{t=0}^T(t+s)\delta_t
                       +\frac{3L\beta^2\sigma^2}{\lambda^2}\\
 &\le\frac{\lambda s(s-1)R^2}{W_T}
           +\frac{4(T+1)S^2}{\lambda W_T}
           +\frac{3L\beta^2\sigma^2}{\lambda^2}.
\end{align*}
Substitute $W_T=(T+1)(T+2s)/2$, $S^2\le9\sigma^2/4$, and
$\beta=1/(T+s+1)$ to obtain the sharper three-term bound
\begin{equation}
 \E[F(\bar x_T)-F^*]
 \le\frac{2\lambda s(s-1)R^2}{(T+1)(T+2s)}
       +\frac{18\sigma^2}{\lambda(T+2s)}
       +\frac{3\kappa\sigma^2}{\lambda(T+s+1)^2}.
 \label{eq:sc-mean-exact-app}
\end{equation}
For the simpler bound, use $s\ge4\kappa\ge4$:
\begin{align*}
 \frac{18\sigma^2}{\lambda(T+2s)}
 +\frac{3\kappa\sigma^2}{\lambda(T+s+1)^2}
 &\le\left(18+\frac{3\kappa}{T+s+1}\right)
                    \frac{\sigma^2}{\lambda(T+s+1)}\\
 &\le\left(18+\frac34\right)\frac{\sigma^2}{\lambda(T+s+1)}
 \le\frac{19\sigma^2}{\lambda(T+s+1)}.
\end{align*}
If $J$ is independent of the run with
$\Pr(J=t)=(t+s)/W_T$, the weighted average of expected STORM function gaps equals
$\E[F(x_J)-F^*]$. The same calculation therefore proves the
random-output assertion.
\end{proof}

\subsection{A block comparison for an SGD companion}
\begin{lemma}[From an earlier block to the last point]
\label{lem:blocks}
Let $F$ be differentiable and convex, and consider
\[
 y_{t+1}=y_t-\eta_t h_t,\qquad
 \E[h_t\mid\mathcal F_t]=\nabla F(y_t),
\]
where $(\mathcal F_t)$ is a filtration and $y_t$ is
$\mathcal F_t$-measurable. Let $T\ge1$ be an integer and set
\[
 M=\lceil (T+1)/4\rceil,\qquad K=\lceil (T+1)/2\rceil,\qquad
 \gamma_t=\frac{c}{t+s+1},\qquad c>0,\ s\ge0.
\]
Use the boundaries in \eqref{eq:rounded-boundaries} and the stepsizes
$\eta_t=2^{-i}\gamma_t$ for $U_i\le t<U_{i+1}$.
Assume the displayed function-value expectations are finite.
If $\E\norm{h_t}^2\le G^2$ for $M-1\le t\le T$, then
\begin{equation}
 \E F(y_T)
 \le\min_{M-1\le t\le K-1}\E F(y_t)+64G^2\gamma_T.
 \label{eq:block-result}
\end{equation}
It is enough to introduce $y_{T+1}$ virtually; the output is $y_T$.
\end{lemma}

\begin{proof}
Since $T\ge1$, we have $m\ge1$. We use the geometry proved in
Lemma~\ref{lem:rounded-geometry}:
$M\ge (T+1)/4$, $K\le2M$, and $n=K-M+1\ge (T+1)/4$.

\emph{Comparison with an earlier point.}
Fix deterministic indices $M-1\le p\le q\le T$.
For every $p\le t\le q$, expand the update around $y_p$:
\[
 \norm{y_{t+1}-y_p}^2
 =\norm{y_t-y_p}^2-2\eta_t\inner{h_t}{y_t-y_p}
                       +\eta_t^2\norm{h_t}^2.
\]
Although $y_p$ is random, it is $\mathcal F_t$-measurable because
$p\le t$. The same is true of $y_t$. Consequently,
\begin{align*}
 \E[\inner{h_t}{y_t-y_p}\mid\mathcal F_t]
 &=\inner{\E[h_t\mid\mathcal F_t]}{y_t-y_p}\\
 &=\inner{\nabla F(y_t)}{y_t-y_p}\\
 &\ge F(y_t)-F(y_p).
\end{align*}
The last inequality is convexity applied with expansion point $y_t$:
\[
 F(y_p)\ge F(y_t)+\inner{\nabla F(y_t)}{y_p-y_t}.
\]
Taking expectations in the distance identity therefore gives
\[
 2\eta_t\bigl(\E F(y_t)-\E F(y_p)\bigr)
 \le \E\norm{y_t-y_p}^2-\E\norm{y_{t+1}-y_p}^2
                +\eta_t^2G^2.
\]
Summing from $t=p$ to $q$ cancels intermediate distance.
The first distance is $\norm{y_p-y_p}^2=0$, so
\begin{align*}
 2\sum_{t=p}^q\eta_t\bigl(\E F(y_t)-\E F(y_p)\bigr)
 \le-\E\norm{y_{q+1}-y_p}^2
                          +G^2\sum_{t=p}^q\eta_t^2\le G^2\sum_{t=p}^q\eta_t^2.
\end{align*}
We will use the resulting inequality in the form
\begin{equation}
 2\sum_{t=p}^q\eta_t\bigl(\E F(y_t)-\E F(y_p)\bigr)
 \le G^2\sum_{t=p}^q\eta_t^2.
 \label{eq:reference-comparison}
\end{equation}
All distances in this calculation have finite expectations: indeed,
$y_t-y_p=-\sum_{j=p}^{t-1}\eta_jh_j$ is a finite sum of random vectors
with finite second moments.

\emph{Comparison of adjacent blocks.}
Use the comparison blocks \eqref{eq:comparison-blocks}. They are
consecutive from $M-1$ through $T$, and $B_m=\{T\}$.
Choose index $\tau_i\in B_i$ minimizing
$\E F(y_t)$ on $B_i$, and write
\[
 H_i=\E F(y_{\tau_{i+1}})-\E F(y_{\tau_i})
 \qquad(0\le i<m).
\]
The index $\tau_i$ is deterministic because the quantities being
minimized are expectations, not observed sample-path values.

Apply \eqref{eq:reference-comparison} with
$p=\tau_i$ and $q=\max B_{i+1}$. The left sum consists of the
remaining part of the source block $B_i$ and all of the target block
$B_{i+1}$. On the source part, the choice of $\tau_i$ gives
\[
 \E F(y_t)-\E F(y_{\tau_i})\ge0.
\]
On the target block, the choice of $\tau_{i+1}$ gives
\[
 \E F(y_t)-\E F(y_{\tau_i})
 \ge\E F(y_{\tau_{i+1}})-\E F(y_{\tau_i})=H_i.
\]
All stepsizes are positive. Thus the whole left side of
\eqref{eq:reference-comparison} is at least
$2H_i\sum_{t\in B_{i+1}}\eta_t$. Dividing by this positive
weight sum gives
\begin{equation}
 H_i\le
 \frac{G^2\sum_{t=\tau_i}^{\max B_{i+1}}\eta_t^2}
      {2\sum_{t\in B_{i+1}}\eta_t}.
 \label{eq:adjacent-blocks}
\end{equation}
This reasoning is valid whether $H_i$ is positive, zero, or negative.

\emph{The effect of rounded block lengths.}
First consider $i=0$. By \eqref{eq:rounded-block-ratios}, the numerator
in \eqref{eq:adjacent-blocks} contains at most $3|B_1|$ terms.
Furthermore, $T+1\le4M$ and $s\ge0$ give
\[
 \eta_t\le\gamma_{M-1}=\frac{c}{M+s}
                    \le4\gamma_T
 \quad\text{throughout the first comparison span}.
\]
On $B_1$, the multiplier is $1/2$, and $t\le T$ implies
$\gamma_t\ge\gamma_T$. Each target stepsize is therefore
at least $\gamma_T/2$. Substituting these upper and lower
bounds into \eqref{eq:adjacent-blocks} yields
\[
 H_0\le
 \frac{3|B_1|G^2(4\gamma_T)^2}
      {2|B_1|(\gamma_T/2)}
 =48G^2\gamma_T.
\]

Next take $1\le i<m$. By \eqref{eq:rounded-block-ratios}, the
comparison span contains at most $4|B_{i+1}|$ indices.
All its indices satisfy $t\ge U_1=K\ge (T+1)/2$, whence
\[
 \gamma_t=\frac{c}{t+s+1}\le2\gamma_T.
\]
The largest multiplier in this span is $2^{-i}$, so its largest
stepsize is at most $2^{1-i}\gamma_T$.
On the target block the multiplier is $2^{-i-1}$, and every
stepsize is at least $2^{-i-1}\gamma_T$. Consequently,
\[
 H_i\le
 \frac{4|B_{i+1}|G^2(2^{1-i}\gamma_T)^2}
      {2|B_{i+1}|2^{-i-1}\gamma_T}
 =16\,2^{-i}G^2\gamma_T.
\]
When $m=1$, there are no comparisons of this second type.

\emph{Summation over the blocks.}
Since $\tau_m=T$, consecutive differences telescope:
\begin{align*}
 \E F(y_T)-\E F(y_{\tau_0})
 &=\sum_{i=0}^{m-1}H_i\\
 &\le\left(48+16\sum_{i=1}^{m-1}2^{-i}\right)
                                      G^2\gamma_T\\
 &\le64G^2\gamma_T.
\end{align*}
The defining property of $\tau_0$ now proves
\eqref{eq:block-result}. The last comparison uses the update at
$t=T$ only to produce the nonnegative terminal distance
$\E\norm{y_{T+1}-y_{\tau_{m-1}}}^2$, which is discarded.
The point $y_{T+1}$ need not be returned or computed by the algorithm.
\end{proof}

\subsection{Proof of the last-iterate theorem}
\begin{proof}[Proof of Theorem~\ref{thm:strong}]
Apply Lemma~\ref{lem:companion}. In particular, the companion
starts at $y_0=x_0$ and has centered innovations satisfying
\[
 \E[\nu_t\mid\mathcal F_t]=0,\qquad
 \E[\norm{\nu_t}^2\mid\mathcal F_t]\le S^2,
 \qquad S^2=(1+\beta)^2\sigma^2\le\frac94\sigma^2.
\]
Write
\[
 D_t=\E\norm{y_t-x^*}^2,\qquad
 \delta_t=\E[F(y_t)-F^*],\qquad
 \gamma_t=\frac2{\lambda(t+s+1)},
\]
and let $M=\lceil (T+1)/4\rceil$, $K=\lceil (T+1)/2\rceil$.
The schedule \eqref{eq:sctaper} satisfies $\eta_t=\gamma_t$ for $0\le t<K$,
and $\eta_t\le\gamma_t$ everywhere. Since
$s\ge4\kappa^2\ge4\kappa$, every actual and virtual stepsize
is at most $1/(2L)$.

\emph{Distance decrease in the first half.}
Expand one companion update around $x^*$:
\[
 y_{t+1}-x^*
 =y_t-x^*-\eta_t\nabla F(y_t)-\eta_t\nu_t.
\]
The first three terms are $\mathcal F_t$-measurable, so the
cross term with $\nu_t$ vanishes after taking expectation.
Gradient-map contraction and $\nabla F(x^*)=0$ then imply
\begin{align}
 D_{t+1}
 &\le(1-\lambda\eta_t)^2D_t+\eta_t^2S^2\nonumber\\
 &\le(1-\lambda\eta_t)D_t+\eta_t^2S^2.
 \label{eq:last-distance}
\end{align}
The second line follows from $0\le\lambda\eta_t\le1$.
During the first block, where $\eta_t=\gamma_t$, we claim that
\begin{equation}
 D_t\le\frac{(s+1)^2R^2}{(t+s+1)^2}
                  +\frac{4S^2}{\lambda^2(t+s+1)},\qquad 0\le t\le K.
 \label{eq:burn-in}
\end{equation}
At $t=0$, the first term equals $D_0=R^2$ and the second is
nonnegative. For the induction step, put $u=t+s+1$.
Here $u\ge5$, so $1-2/u\ge0$ and substitution of the induction
upper bound into \eqref{eq:last-distance} preserves the inequality.
Since $\eta_t=2/(\lambda u)$, this gives
\begin{align*}
 D_{t+1}
 &\le\left(1-\frac2u\right)
       \left(\frac{(s+1)^2R^2}{u^2}
                    +\frac{4S^2}{\lambda^2u}\right)
                    +\frac{4S^2}{\lambda^2u^2}\\
 &=\frac{(s+1)^2R^2(u-2)}{u^3}
                    +\frac{4S^2(u-1)}{\lambda^2u^2}\\
 &\le\frac{(s+1)^2R^2}{(u+1)^2}
                    +\frac{4S^2}{\lambda^2(u+1)}.
\end{align*}
The last line follows term by term from
\[
 (u-2)(u+1)^2=u^3-3u-2\le u^3,
 \qquad (u-1)(u+1)=u^2-1\le u^2.
\]
This proves \eqref{eq:burn-in}.

\emph{Second moments of the search directions.}
For any $t\ge M-1$, discard contraction in
\eqref{eq:last-distance} and sum from $M-1$ to $t-1$:
\[
 D_t\le D_{M-1}+S^2\sum_{j=M-1}^{t-1}\eta_j^2
       \le D_{M-1}+S^2\sum_{j=M-1}^{\infty}\gamma_j^2.
\]
The infinite sum is only an upper bound on a finite sum; it does
not introduce additional algorithmic iterations. By monotonicity
of the summand and $M+s\ge1$,
\begin{align*}
 \sum_{j=M-1}^{\infty}\gamma_j^2
 &=\frac4{\lambda^2}\sum_{k=M}^{\infty}\frac1{(k+s)^2}\\
 &\le\frac4{\lambda^2}
       \left(\frac1{(M+s)^2}
                         +\int_M^\infty\frac{du}{(u+s)^2}\right)\\
 &=\frac4{\lambda^2}\left(\frac1{(M+s)^2}+\frac1{M+s}\right)
 \le\frac8{\lambda^2(M+s)}.
\end{align*}
Use \eqref{eq:burn-in} at $t=M-1$ to conclude that
\begin{equation}
 D_t\le\frac{(s+1)^2R^2}{(M+s)^2}
                    +\frac{12S^2}{\lambda^2(M+s)},
                    \qquad M-1\le t\le T.
 \label{eq:tail-distance}
\end{equation}

Now write $h_t=\nabla F(y_t)+\nu_t$ for the companion direction.
Since $y_t$ is $\mathcal F_t$-measurable, conditional centering gives
$\E\inner{\nabla F(y_t)}{\nu_t}=0$.
Standard smoothness and $\nabla F(x^*)=0$ imply
$\norm{\nabla F(y_t)}\le L\norm{y_t-x^*}$.
Consequently,
\begin{align*}
 \E\norm{h_t}^2
 &=\E\norm{\nabla F(y_t)}^2+\E\norm{\nu_t}^2\\
 &\le L^2D_t+S^2\\
 &\le\frac{L^2(s+1)^2R^2}{(M+s)^2}
        +\left(1+\frac{12\kappa^2}{M+s}\right)S^2\\
 &\le\frac{L^2(s+1)^2R^2}{(M+s)^2}+4S^2=:G^2.
\end{align*}
For the last line, $M+s\ge s\ge4\kappa^2$ gives
$12\kappa^2/(M+s)\le3$. Thus the moment assumption of
Lemma~\ref{lem:blocks} holds on the required trajectory segment.
This argument uses smoothness of $F$ only; it imposes neither
sample smoothness nor a globally bounded stochastic gradient.

\emph{A reference function value.}
For $M-1\le t\le K-1$, the stepsize is $\gamma_t$.
Lemma~\ref{lem:convex-distance} applies because $\gamma_t\le1/(2L)$,
and gives
\[
 D_{t+1}\le D_t-\gamma_t\delta_t+\gamma_t^2S^2.
\]
Rearranging and summing makes the interior distances cancel:
\[
 \sum_{t=M-1}^{K-1}\gamma_t\delta_t
 \le D_{M-1}-D_K+S^2\sum_{t=M-1}^{K-1}\gamma_t^2
 \le D_{M-1}+S^2\sum_{t=M-1}^{K-1}\gamma_t^2.
\]
The left side is at least the smallest $\delta_t$ in the range
times the positive sum of weights. Hence
\[
 \min_{M-1\le t\le K-1}\delta_t
 \le\frac{D_{M-1}+S^2\sum_{t=M-1}^{K-1}\gamma_t^2}
              {\sum_{t=M-1}^{K-1}\gamma_t}.
\]
There are $n=K-M+1\ge (T+1)/4$ indices. Their largest denominator
in $\gamma_t$ is $K+s$ and their smallest is $M+s$, so
\[
 \sum_{t=M-1}^{K-1}\gamma_t\ge\frac{2n}{\lambda(K+s)},
 \qquad
 \sum_{t=M-1}^{K-1}\gamma_t^2\le\frac{4n}{\lambda^2(M+s)^2}.
\]
Also $K\le2M$, and therefore $K+s\le2(M+s)$. It follows that
\begin{align*}
 \min_{M-1\le t\le K-1}\delta_t
 &\le\frac{\lambda(K+s)}{2n}D_{M-1}
          +\frac{2S^2(K+s)}{\lambda(M+s)^2}\\
 &\le\frac{4\lambda(M+s)}{T+1} D_{M-1}
          +\frac{4S^2}{\lambda(M+s)}\\
 &\le\frac{4\lambda(s+1)^2R^2}{(T+1)(M+s)}
          +\frac{16S^2}{\lambda (T+1)}+\frac{4S^2}{\lambda(M+s)}\\
 &\le\frac{4\lambda(s+1)^2R^2}{(T+1)(M+s)}+\frac{32S^2}{\lambda (T+1)}.
\end{align*}
In the third line, we use the sharper bound
\eqref{eq:burn-in} at $M-1$, whose noise coefficient is $4$,
rather than the tail bound with coefficient $12$.
For the last line, use $M+s\ge (T+1)/4$.

\emph{The last companion iterate.}
Lemma~\ref{lem:blocks}, with $c=2/\lambda$, adds at most
\begin{align*}
 64G^2\gamma_T
 &=\frac{128L^2(s+1)^2R^2}{\lambda(T+s+1)(M+s)^2}
                        +\frac{512S^2}{\lambda(T+s+1)}\\
 &\le\frac{32\lambda(s+1)^2R^2}{(T+1)(M+s)}
                        +\frac{512S^2}{\lambda (T+1)}.
\end{align*}
For the initialization term, divide its coefficient by the
coefficient of $\lambda(s+1)^2R^2/[(T+1)(M+s)]$:
\[
 \frac{128\kappa^2(T+1)}{(T+s+1)(M+s)}
 =128\frac{\kappa^2}{M+s}\frac{T+1}{T+s+1}\le32.
\]
Combining this with the preceding bound on the minimum over
$M-1\le t\le K-1$, and using $4+32=36$ and
$32+512=544$, gives
\begin{equation}
 \E[F(y_T)-F^*]
 \le\frac{36\lambda(s+1)^2R^2}{(T+1)(M+s)}
                        +\frac{544S^2}{\lambda (T+1)}.
 \label{eq:last-companion-bound}
\end{equation}

\emph{Transfer to STORM.}
The comparison bound \eqref{eq:expected-transfer-app} gives
\[
 \E[F(x_T)-F^*]
 \le2\E[F(y_T)-F^*]
                  +\frac{3L\beta^2\sigma^2}{\lambda^2}.
\]
Substitute \eqref{eq:last-companion-bound}, then use
$S^2\le9\sigma^2/4$ and $\beta=1/(T+s+1)$:
\begin{align*}
 \E[F(x_T)-F^*]
 &\le\frac{72\lambda(s+1)^2R^2}{(T+1)(M+s)}
        +\frac{1088S^2}{\lambda (T+1)}
        +\frac{3\kappa\sigma^2}{\lambda(T+s+1)^2}\\
 &\le\frac{72\lambda(s+1)^2R^2}{(T+1)(M+s)}
        +\frac{2448\sigma^2}{\lambda (T+1)}
        +\frac{3\kappa\sigma^2}{\lambda(T+s+1)^2}.
\end{align*}
Finally, $s\ge4\kappa^2\ge\kappa$ and
$(T+s+1)^2\ge4(T+1)s\ge\kappa(T+1)$ imply
\[
 \frac{3\kappa\sigma^2}{\lambda(T+s+1)^2}
 \le\frac{3\sigma^2}{\lambda (T+1)}.
\]
Finally, \(M+s\ge(T+1)/4\), so \(72/[(T+1)(M+s)]\le288/(T+1)^2\), which proves the bound in Theorem~\ref{thm:strong}.
\end{proof}

\end{document}